\documentclass[12pt,oneside]{article}
\usepackage[utf8]{inputenc}
\usepackage[margin=100pt]{geometry}
\usepackage{enumitem}
\usepackage{verbatim}
\usepackage{microtype}
\usepackage{physics}

\usepackage[pdfusetitle]{hyperref}
\usepackage{color}
\hypersetup{colorlinks,linkcolor=blue,citecolor=cyan}
\usepackage{amsthm}
\usepackage[nameinlink]{cleveref}
\usepackage{amsmath,amssymb}

\usepackage{tikz}
\usepackage{tikz-cd}
\usepackage{xparse,etoolbox}

\newcommand{\A}{\mc A}
\newcommand{\ALG}{\mathrm{ALG}}
\newcommand{\amp}{\mathrel \&}
\newcommand{\ap}{\mathrm{ap}}
\newcommand{\B}{\mc B}
\newcommand{\C}{\mc C}
\newcommand{\car}{\curvearrowright}
\newcommand{\cb}{\mathrm{cb}}
\newcommand{\concat}{\string^}
\newcommand{\diff}{\mathrm d}
\newcommand{\E}{\mathrel E}
\newcommand{\EINV}{\mathrm{EINV}}
\newcommand{\F}{\mr F}
\newcommand{\G}{\mc G}
\newcommand{\hra}{\hookrightarrow}
\newcommand{\INV}{\mathrm{INV}}
\newcommand{\K}{\mc K}
\newcommand{\mb}{\mathbb}
\newcommand{\mc}{\mathcal}
\newcommand{\mr}{\mathrel}
\newcommand{\N}{\mb N}
\newcommand{\NULL}{\mathrm{NULL}}

\newcommand{\p}{\mathbf p}
\newcommand{\R}{\mb R}
\newcommand{\thra}{\twoheadrightarrow}
\newcommand{\tl}{\mr\triangleleft}
\newcommand{\uhr}{\upharpoonright}
\newcommand{\wt}{\widetilde}
\newcommand{\Z}{\mb Z}
\DeclareMathOperator{\Aut}{Aut}
\DeclareMathOperator{\cod}{cod}
\DeclareMathOperator{\dom}{dom}
\DeclareMathOperator{\End}{End}
\DeclareMathOperator{\Fin}{Fin}
\DeclareMathOperator{\id}{id}
\DeclareMathOperator{\im}{im}
\DeclareMathOperator{\Inn}{Inn}
\DeclareMathOperator{\len}{len}
\DeclareMathOperator{\Out}{Out}
\DeclareMathOperator{\SL}{SL}
\DeclareMathOperator{\Stab}{Stab}
\DeclareMathOperator{\Sym}{Sym}

\newtheorem{thm}{Theorem}[section]
\newtheorem{cor}[thm]{Corollary}
\newtheorem{lem}[thm]{Lemma}
\newtheorem{prop}[thm]{Proposition}

\theoremstyle{definition}
\newtheorem{defn}[thm]{Definition}
\newtheorem{eg}[thm]{Example}
\newtheorem{prob}[thm]{Problem}
\newtheorem{rmk}[thm]{Remark}

\newcommand{\blocktheorem}[1]{%
  \csletcs{old#1}{#1}
  \csletcs{endold#1}{end#1}
  \RenewDocumentEnvironment{#1}{o}
    {\par\addvspace{1.5ex}
     \noindent\begin{minipage}{\textwidth}
     \IfNoValueTF{##1}
       {\csuse{old#1}}
       {\csuse{old#1}[##1]}}
    {\csuse{endold#1}
     \end{minipage}
     \par\addvspace{1.5ex}}
}

\blocktheorem{thm}
\blocktheorem{cor}
\blocktheorem{lem}
\blocktheorem{prop}
\blocktheorem{defn}
\blocktheorem{eg}
\blocktheorem{prob}
\blocktheorem{rmk}

\title{Lifts of Borel actions on quotient spaces}

\author{Joshua R. Frisch, Alexander S. Kechris and Forte Shinko}

\date{}

\begin{document}

\maketitle

\centerline{Dedicated to Benjamin Weiss on his 80th birthday}

\begin{abstract}
    Given a countable Borel equivalence relation $E$
    and a countable group $G$,
    we study the problem of when a Borel action of $G$ on $X/E$
    can be lifted to a Borel action of $G$ on $X$.
\end{abstract}

\tableofcontents

\section{Introduction}

\subsection{Automorphisms of equivalence relations}

A \textbf{countable Borel equivalence relation (CBER)}
is an equivalence relation $E$ on a standard Borel space $X$
such that $E$ is Borel when considered as a subset of $X^2$.
Let $\pi_E:X\to X/E$ denote the quotient map.

Let $E$ be a CBER on $X$.
The \textbf{automorphism group} of $E$,
denoted $\Aut_B(E)$
(or $N_B[E]$),
is the group of Borel automorphisms of $E$,
that is,
Borel automorphisms $T:X\to X$ such that
$x\E y \iff T(x)\E T(y)$,
under composition.
The \textbf{inner automorphism group} of $E$
(or the \textbf{full group} of $E$),
denoted $\Inn_B(E)$ (or $[E]_B$),
is the normal subgroup of $\Aut_B(E)$ consisting of the $T\in\Aut_B(E)$
such that $x\E T(x)$.
The normalizer of $\Inn_B(E)$ in the group of Borel automorphisms of $X$ is $\Aut_B(E)$.
By a result of Miller and Rosendal \cite[Proposition 2.1]{MR07},
if $E$ is aperiodic,
then the natural map $\Aut_B(E)\to \Aut(\Inn_B(E))$
is an isomorphism.
The \textbf{outer automorphism group} of $E$,
denoted $\Out_B(E)$,
is the quotient group $\Aut_B(E)/\Inn_B(E)$.

Let $E$ and $F$ be CBERs on $X$ and $Y$ respectively.
A function $f:X/E\to Y/F$ is \textbf{Borel}
if the set $\{(x,y)\in X\times Y: f([x]_E) = [y]_F\}$ is Borel,
or equivalently by the Lusin-Novikov theorem
\cite[Theorem 18.10]{Kec95},
if there exists a Borel map $T:X\to Y$ such that $f([x]_E) = [T(x)]_F$.
The \textbf{Borel symmetric group} of $X/E$,
denoted $\Sym_B(X/E)$,
is the set of Borel permutations of $X/E$ under composition.
There is a natural map $\Aut_B(E)\to\Sym_B(X/E)$,
defined by sending $T\in\Aut_B(E)$
to the permutation $[x]_E\mapsto [T(x)]_E$.
This morphism has kernel $\Inn_B(E)$,
so there is a factorization
\[
    \begin{tikzcd}
        \Aut_B(E) \rar[two heads, "p_E"] &
        \Out_B(E) \rar[hook, "i_E"] &
        \Sym_B(X/E)
    \end{tikzcd}.
\]
A Borel permutation of $X/E$ in the image of this morphism
is called an \textbf{outer permutation}.
In other words, $f \in \Sym_B(X/E)$ is outer if
there is $T \in \Aut_B(E)$ such that $f([x]_E) = [T(x)]_E$.

\subsection{Lifts of Borel actions on quotient spaces}

Let $E$ be a CBER on $X$ and let $G$ be a countable group.
We write $G\car_B (X,E)$
to denote an action of $G$ on $X$ by Borel automorphisms of $E$,
which is equivalent to a morphism $G\to\Aut_B(E)$.
An action $G\car_B (X,E)$ is \textbf{class-bijective}
if $\pi_E$ is class-bijective,
that is,
the restriction of $\pi_E$ to every $G$-orbit is an injection,
i.e.,
$g \cdot x \E x \implies g \cdot x = x$.
A \textbf{Borel action} of $G$ on $X/E$,
denoted $G\car_B X/E$,
is an action of $G$ on $X/E$ by Borel permutations,
which is equivalent to a morphism $G\to\Sym_B(X/E)$.
An action $G\car_B X/E$ is \textbf{outer}
if $G$ acts by outer permutations,
or equivalently,
if the morphism $G\to\Sym_B(X/E)$ factors through $i_E$.
Every action $G\car_B (X,E)$ induces
an action $G\car_B X/E$ by composing with $i_E\circ p_E$,
and $\pi_E$ is $G$-equivariant with respect to these actions.
We initiate in this paper the study of the reverse problem:
when does a Borel action $G\car_B X/E$
have a \textbf{lift} to an action $G\car_B (X,E)$?
In other words,
we are interested in the lifting problem
\[
    \begin{tikzcd}
        & \Aut_B(E) \dar[two heads, "p_E"] \\
        & \Out_B(E) \dar[hook, "i_E"] \\
        G \rar \ar[uur, dashed] \urar[dashed] & \Sym_B(X/E)
    \end{tikzcd}
\]
which we will break up into steps by going through $\Out_B(E)$.

\subsection{Main results}

We give in \Cref{sectionBorelActions}
examples of CBERs $E$ that show that even the first step of the lifting problem
 \[
    \begin{tikzcd}
        & \Out_B(E) \dar[hook, "i_E"] \\
        G \rar  \urar[dashed] & \Sym_B(X/E)
    \end{tikzcd}
\]
does not always have a positive solution,
i.e.,
that there are Borel actions $G \car_B X/E$ which are not outer.
In all these examples,
$E$ admits an invariant Borel probability measure
(i.e,
it is generated by a Borel action of a countable group
that has an invariant Borel probability measure).
On the other hand,
we show in \Cref{compressibleCb} that
the full lifting problem has a positive solution,
in a strong sense,
when the CBER $E$ admits no such invariant measure
or equivalently
(by Nadkarni's Theorem)
that it is compressible
(i.e.,
there is a Borel injection that sends every equivalence class
to a proper subset of itself).
\begin{thm}\label{mainCompressible}
    Let $E$ be a compressible CBER.
    Then every Borel action $G\car_B X/E$ has a class-bijective lift $G\car_B (X,E)$.
\end{thm}

This theorem follows from a result
(see \Cref{compressibleLink})
about links
(see \Cref{defnLink})
of pairs $E\subseteq F$ of compressible CBERs
that was also proved
(by a different method)
independently by Ben Miller.
Our proof uses some ideas coming from \cite{FSZ89}.

We do not know if there are non-compressible $E$ that satisfy \Cref{mainCompressible}.
Using this result and a variant of \cite[Corollary 13.3]{KM04},
we show,
in \Cref{genericCb},
that the full lifting problem has a positive solution
generically for an arbitrary \textbf{aperiodic}
(i.e., having all its classes infinite)
CBER $E$.

Below if $G\car_B X/E$,
we let $E^{\vee G}\supseteq E$ be the CBER defined as follows:
\[
    x\E^{\vee G} y
    \iff \exists g\in G (g\cdot [x]_E = [y]_E).
\]
\begin{cor}\label{mainGeneric}
    Let $E$ be an aperiodic CBER on a Polish space $X$.
    Then for any Borel action $G \car_B X/E$,
    there is a comeager $E^{\vee G}$-invariant
    Borel subset $Y\subseteq X$ such that
    $G\car_B Y/E$ has a class-bijective lift.
\end{cor}

In Sections \ref{sectionOuter}-\ref{sectionAmenable},
we study the lifting problem for outer actions.
A lift of an outer action is a solution to the following lifting problem:
\[
    \begin{tikzcd}
        & \Aut_B(E) \dar[two heads, "p_E"] \\
        G \rar \urar[dashed] & \Out_B(E)
    \end{tikzcd}
\]

Below we use the following terminology.
If a group $G$ acts on a set $X$,
we denote by $E_G^X$ the induced equivalence relation
whose classes are the $G$-orbits.
An action of group $G$ on a set $X$ is \textbf{free}
if for any $g \neq 1$ and $x\in X$,
$g\cdot x \neq x$.
If the set $X$ carries a measure and the action is measure-preserving,
we only require that this holds for almost all $x$.
A Borel action of a countable group $G$
on a standard Borel space $X$ is \textbf{pmp}
if it has an invariant Borel probability measure.
A countable group $G$ is \textbf{treeable}
if it admits a free, pmp Borel action on a standard Borel space $X$
such that the induced CBER $E_G^X$ is treeable,
i.e.,
its classes are the connected components of an acyclic Borel graph on $X$.
For example,
all amenable and free groups are treeable but all property (T) groups
and all products of an infinite group with a non-amenable group are not treeable.

We now have the following results
(see \Cref{amenableCb}, \Cref{lfCb} for \ref{mainCb},
and \Cref{finiteAmalgamLift}, \Cref{amenableLift} for \ref{mainLift}).
Below a CBER is \textbf{smooth} if
it admits a Borel set meeting every class in exactly one point.

\begin{thm}\label{mainLifts}
    \leavevmode
    \begin{enumerate}[label=(\arabic*)]
        \item \label{mainCb}
            Every outer action of any abelian group,
            and in fact any group for which
            the conjugacy equivalence relation
            on its space of subgroups is smooth,
            and any locally finite group has a class-bijective lift.
        \item \label{mainLift}
            Every outer action of any amenable group
            and any amalgamated free product of finite groups has a lift.
    \end{enumerate}
\end{thm}

The proof of \Cref{mainLifts}, \ref{mainLift} for the case of amenable groups
makes use of the quasi-tiling machinery developed in the work of Ornstein and Weiss
\cite{OW80}, \cite{OW87} and also uses some ideas from \cite{FSZ89}.
Also the proof of \Cref{mainLifts}, \ref{mainLift}
for the case of amalgamated free products of finite groups
also uses some ideas from \cite{Tse13}.
We do not know if the conclusion of \ref{mainLift}
can be restrengthened to having a class-bijective lift.

On the other hand we have an upper bound for groups that have this lifting property
(see \Cref{liftTreeable}).
The proof of the next result is motivated by \cite{CJ85} and \cite{FSZ89}.
\begin{prop}\label{mainLiftTreeable}
    If every outer action of a countable group $G$ lifts, then $G$ is treeable.
\end{prop}

We do not know a characterization of the class of countable groups
all of whose outer actions have a lift or a class-bijective lift.
\Cref{sectionSummary} contains a summary of what we know
about the classes of groups all of whose outer actions have a lift
(resp., a class-bijective lift).

\subsection*{Acknowledgements}
We would like to thank Aristotelis Panagiotopoulos,
for getting us to think about these problems
by asking whether (in our terminology)
every action $G\car_B X/E$ is outer,
Ben Miller for many valuable comments and suggestions,
and Adrian Ioana for helpful discussions concerning
\Cref{probOuterCBER} \ref{outerNoncompressible}.
Research was partially supported by NSF Grant DMS-1950475.

\section{Preliminaries}

\subsection{Countable Borel equivalence relations}

We review here some basic notions and results that we will use in the sequel.
A general reference is the survey paper \cite{Kec22}.
Given a CBER $E$ on $X$,
we denote for each $A\subseteq X$ by
$[A]_E = \{x\in X : \exists y \in A \,(x \E y)\}$
the \textbf{$E$-saturation} of $A$.
In particular if $x\in X$,
$[\{x\}]_E = [x]_E$ is the equivalence class of $E$.
Dually the \textbf{$E$-hull} of $A$ is the set
$\{x\in X : [x]_E\subseteq A\}$.
Finally we let $E \uhr A = E\cap A^2$ be the restriction of $E$ to $A$.
A set $A\subseteq X$ is \textbf{$E$-invariant} if $A=[A]_E$.
For each set $S$,
we denote by $\Delta_S$ the equality relation on $S$
and we also let $I_S= S^2$.

For CBERs $E,F$ on $X,Y$ resp.,
we denote by $E\oplus F$ the \textbf{direct sum} of $E,F$.
Formally this is the equivalence relation on the direct sum $X\sqcup Y$ of $X,Y$
which agrees with $E$ on $X$ and with $F$ on $Y$.
Similarly we define the direct sum $\bigoplus_n E_n$ for a sequence $(E_n)$ of CBERs.
The \textbf{product} of $E,F$ is the equivalence relation on $X\times Y$
given by $(x,y) \mathrel{E\times F}  (x', y') \iff (x \E x') \amp (y \F y')$.

If $E, F$ are CBERs on $X$ and $E\subseteq F$ (as sets of ordered pairs),
then $E$ is a \textbf{subequivalence relation} of $F$
and $F$ is an \textbf{extension} of $E$.
If every $F$-class contain only finitely many $E$-classes,
we say that $F$ has \textbf{finite index} over $E$
and if for some $N$ every $F$-class contains at most $N$ $E$-classes,
we say that $F$ has \textbf{bounded index} over $E$.
If every $F$-class contains exactly $N$ $E$-classes we write $[F:E] = N$.
Finally, $E\vee F$ is the smallest equivalence relation containing $E$ and $F$.

A \textbf{complete section} of a CBER $E$ on $X$ is
a set $S\subseteq X$ that meets every $E$-class.
A \textbf{transversal} of $E$ is a subset $T\subseteq X$
that meets every $E$-class in exactly one point.
If a Borel transversal exists,
we say that $E$ is \textbf{smooth}.
A CBER $E$ is \textbf{finite} if every $E$-class is finite
and it is \textbf{hyperfinite} if $E=\bigcup_n E_n$,
where $E_n \subseteq E_{n+1}$ and $E_n$ is finite, for each $n$.
A canonical non-smooth hyperfinite CBER is $E_0$ on $2^\N$ defined by
$x \E_0 y \iff \exists m \,\forall n \ge m \,(x_n = y_n)$.
We say that a CBER $E$ is \textbf{aperiodic} if every $E$-class is infinite.
For any CBER $E$ there is a unique decomposition $X = A\sqcup B$
into $E$-invariant Borel sets such that $E\uhr A$ is finite and $E\uhr B$ is aperiodic.
These are, resp., the finite and infinite parts of $E$.
A CBER $E$ on $X$ is \textbf{treeable}
if there is an acyclic Borel graph $\Gamma\subseteq X^2$
whose connected components are exactly the $E$-classes.
Every hyperfinite CBER is treeable.

A CBER $E$ on $X$ is \textbf{compressible}
if there is a Borel injection $T\colon X \to X$
such that $T([x]_E) \subsetneqq [x]_E$, for each $x$.
A Borel set $A\subseteq X$ is ($E$-)compressible
if $E\uhr A$ is compressible.
In that case $[A]_E$ is compressible as well and
there is a Borel injection $T\colon X\to X$ such that $T(x) \E x$,
for every $x$,
and $T([A]_E) = A$;
see \cite[Proposition 3.26]{Kec22}.
Recall also from \cite[Proposition 3.23]{Kec22} that $E$ is compressible
iff $E\cong_B E\times I_\N$
(where for two CBERs $F_1, F_2$ on $X_1, X_2$,
resp., $F_1\cong_B F_2$ means that they are \textbf{Borel isomorphic},
i.e.,
there is a Borel bijection $T\colon X_1\to X_2$ that takes $F_1$ to $F_2$)
and also $E$ is compressible
iff it contains a smooth, aperiodic subequivalence relation.

Given CBERs $E,F$ on $X,Y$,
resp.,
we say that $E$ is \textbf{Borel reducible} to $F$,
in symbols $E\le_B F$,
if there is a Borel map $T\colon X \to Y$
such that $x \E x' \iff T(x) \F T(x')$.
Such a $T$ is called a \textbf{reduction} of $E$ to $F$.
Moreover $E,F$ are \textbf{Borel bireducible},
in symbols $E\sim_B F$,
if $(E\le_B F) \amp (F\le_B E)$.
We have that $E\sim_B F$ iff
there is a Borel bijection $T\colon X/E \to Y/F$;
see \cite[Theorem 3.32]{Kec22}.

Given a countable group $G$ and a Borel action of $G$ on $X$,
denote by $E_G^X$ the CBER induced by this action,
i.e.,
the equivalence relation whose classes are exactly the orbits of this action.
The Feldman-Moore Theorem
(see, e.g., \cite[Theorem 3.3]{Kec22})
asserts that for every CBER $E$ on $X$
there is a countable group $G$ and a Borel action of  $G$ on $X$ such that $E= E_G^X$.

By a \textbf{partial subequivalence relation} of a CBER $E$ on $X$,
we mean an equivalence relation $F$ on a subset $A\subseteq X$ such that $F\subseteq E$.
A Borel finite partial subequivalence relation is abbreviated as \textbf{fsr}.

Let now $X$ be a standard Borel space and denote by $[X]^{<\infty}$
the standard Borel space of finite subsets of $X$.
If $E$ is a CBER on $X$,
we denote by $[E]^{<\infty}$ the subset of $[X]^{<\infty}$
consisting of all finite sets that are contained in a single $E$-class.
Then $[E]^{<\infty}$ is Borel.
For each set $\Phi\subseteq [E]^{<\infty}$,
an fsr $F$ of $E$ defined on the set $A\subseteq X$ is \textbf{$\Phi$-maximal},
if every $F$-class is in $\Phi$ and every finite set $S$ disjoint from $A$ is not in $\Phi$.
We now have the following result;
see \cite[Lemma 7.3]{KM04}:
If $E$ is a CBER and $\Phi\subseteq [E]^{\infty}$ is Borel,
then there is a Borel $\Phi$-maximal fsr of $E$.
The \textbf{intersection graph} of $E$ is the graph on $[E]^{<\infty}$,
where $S,T$ are connected by an edge iff
there are distinct and have nonempty intersection.
The proof of \cite[Lemma 7.3]{KM04} uses the fact that
this graph has a countable Borel coloring,
i.e.,
a Borel map $c\colon [E]^{<\infty}\to \N$,
which is a coloring of this graph.
 
For each CBER $E$ on $X$,
denote by $\INV_E$
the standard Borel space of invariant Borel probability measures on $X$,
i.e.,
the Borel probability measures on $X$ for which there is a Borel,
measure-preserving action of a countable group $G$ on $X$ with $E_G^X = E$.
We also let $\EINV_E$ be the Borel subset of $\INV_E$
consisting of all ergodic measures in $\INV_E$.
Nadkarni's Theorem
(see \cite[Theorem 5.6]{Kec22})
states that $E$ is compressible iff $\INV_E$ is empty.
The Ergodic Decomposition Theorem of Farrell and Varadarajan
(see \cite[Theorem 5.12]{Kec22})
asserts that if  $\INV_E \neq \varnothing$,
then there is a Borel surjection $\pi \colon X \thra \EINV_E$ such that
\begin{enumerate}[label=(\roman*)]
    \item \label{decompositionFirst}
        $\pi$ is $E$-invariant;
    \item
        If $X_e = \pi^{-1}(\{e\})$,
        for $e\in \EINV_E$,
        then $e(X_e) = 1$ and $e$ is the unique $E$-invariant probability measure
        concentrating on $X_e$;
    \item \label{decompositionLast}
        If $\mu\in \INV_E$,
        then $\mu = \int \pi(x) \,\diff\mu(x) = \int e \, \diff\pi_*\mu(e)$. 
\end{enumerate}
 
Moreover this map is unique in the following sense:
If $\pi, \pi'$ satisfy \ref{decompositionFirst}-\ref{decompositionLast},
then the set $\{x : \pi(x) \neq \pi'(x)\}$ is compressible. 

The sets $X_e$ are the \textbf{ergodic components} of $E$. 

We say that $E$ is \textbf{uniquely ergodic}
(resp., \textbf{finitely ergodic, countably ergodic})
if $\EINV_E$ is a singleton
(resp., finite, countable).

The Classification Theorem for hyperfinite CBERs
(see \cite[Theorem 8.4]{Kec22})
states that for aperiodic,
non-smooth,
hyperfinite $E,F$,
we have that $E\cong_B F$ iff $\EINV_E$ and $\EINV_F$
have the same cardinality.

\subsection{Cardinal algebras}

A \textbf{cardinal algebra} is a tuple $(A,0,+,\sum)$,
where $(A,0,+)$ is a commutative monoid,
and $\sum{}\colon A^\N \to A$ is an infinitary operation 
satisfying the following axioms:
\begin{enumerate}[label=(\roman*)]
    \item $\sum_i a_i = a_0 + \sum_i a_{i+1}$.
    \item $\sum_i (a_i + b_i) = \sum_i a_i + \sum_i b_i$.
    \item The \textbf{refinement axiom}:
        If $a + b = \sum_i c_i$,
        then there are $(a_i)_i$ and $(b_i)_i$ such that
        $a = \sum_i a_i$,
        $b = \sum_i b_i$
        and $a_i + b_i = c_i$,
    \item The \textbf{remainder axiom}:
        If $(a_i)_i$ and $(b_i)_i$ satisfy $a_i = b_i + a_{i+1}$,
        then there is some $c$ such that $a_i = c + \sum_j b_{i+j}$.
\end{enumerate}
We will need two consequences of these axioms.
For $0\le n\le\infty$,
let $na$ denote the sum of $n$ copies of $a$
(in particular,
let $\infty a$ denote $\sum_i a$).
\begin{enumerate}[label=(\arabic*)]
    \item For any $a,b$,
        \[
            a = a + b \implies a = a + \infty b.
        \]
        To see this,
        use the remainder axiom with $a_i = a$ and $b_i = b$.
        This gives some $c$ such that $a = c + \infty b$.
        Then
        \[
            a + \infty b
            = c + \infty b + \infty b
            = c + \infty b
            = a.
        \]
    \item The \textbf{cancellation law}:
        For any $a,b$ and $0 < n < \infty$,
        \[
            na = nb \implies a = b;
        \]
        see \cite[Theorem 2.34]{Tar49}.
\end{enumerate}

We will need the following cardinal algebras:
\begin{enumerate}[label=(\arabic*)]
    \item The collection of all CBERs up to Borel isomorphism
        is a cardinal algebra under direct sum;
        see \cite[3.C]{KM16}.
    \item Let $E$ be a CBER on $X$.
        We say that $A,B\subseteq X$ are \textbf{$E$-equidecomposable},
        denoted $A\sim_E B$,
        if there is some Borel bijection $T \colon A\to B$
        whose graph is contained in $E$.
        This is an equivalence relation,
        and we denote the class of $A$ by $\wt A$.
        Let $\K(E)$ denote the set of $E$-equidecomposability classes.
        
        Assume now that $E$ is compressible.
        Then for any countable sequence $\wt{A_0}, \wt{A_1}, \ldots$,
        we can assume that the $A_n$ are pairwise disjoint,
        and we can define the infinitary operation as follows:
        \[
            \sum_n\wt{A_n}
            := \wt{\bigcup_n A_n}
        \]
        (We define $+$ analogously,
        and we define $0$ to be the class of the empty set.)
        Then $\K(E)$ with these operations is a cardinal algebra;
        see \cite[Proposition 4.1]{Che21}.
        
        There is an action $\Aut_B(E) \car \K(E)$
        (i.e., a group action preserving $(0, +, \sum)$)
        defined by
        \[
            T \cdot \wt A
            = \wt{T(A)},
        \]
        and this descends to an action $\Out_B(E) \car \K(E)$.
\end{enumerate}

\subsection{Actions on probability spaces}

Let $(X,\mu)$ be a standard probability space,
i.e.,
a standard Borel space with a non-atomic Borel probability measure.
Let $\Aut_\mu(X)$ denote the group of Borel automorphisms
$T:X\to X$ such that $T_*\mu = \mu$,
where $T$ and $T'$ are identified if they agree on a conull set.

Let $E$ be a \textbf{pmp} CBER on $X$,
i.e., a CBER which is generated by
a measure-preserving action of a countable group.
Then $\Aut_\mu(E)$ denotes the set of $T\in\Aut_\mu(X)$
such that $x\E y\iff T(x)\E T(y)$,
for all $x,y$ in a conull subset of $X$.
Let $\Inn_\mu(E)$ denote the normal subgroup of $T\in\Aut_\mu(E)$
such that $x\E T(x)$ for almost every $x\in X$.
Then $\Out_\mu(E)$ denotes the quotient $\Aut_\mu(E)/\Inn_\mu(E)$.

All of the proofs below in the Borel setting
go through mutatis mutandis in the pmp setting.

\section{Borel actions on quotient spaces}
\label{sectionBorelActions}

\subsection{Outer and non-outer actions}
\label{outerNonOuter}
Not every Borel action $G \car_B X/E$ is outer.
For example,
let $2^\N = A\sqcup B$,
where $A$ and $B$ are complete Borel sections for $E_0$
with $\mu(A) \neq \mu(B)$,
where $\mu$ is Lebesgue measure.
Let $E = (E_0\uhr A) \oplus (E_0\uhr B)$.
Then the involution on $X/E$ sending $[x]_{E_0}\cap A$ to $[x]_{E_0}\cap B$ is not outer,
since otherwise we would have $\mu(A) = \mu(B)$.

Note that the following are equivalent:
\begin{enumerate}[label=(\arabic*)]
    \item every Borel action on $X/E$ is outer; 
    \item $i_E$ is a bijection.
\end{enumerate}

This condition is quite strong:
\begin{prop}\label{outerIso}
    Let $G$ be a countable group and let $E$ be a CBER.
    Suppose that every action $G\car_B X/E$ is outer.
    \begin{enumerate}[label=(\arabic*)]
        \item \label{bireducibleIso}
            Whenever $E\cong_B \bigoplus_{g\in G} E_g$,
            with the $E_g$ pairwise Borel bireducible,
            then the $E_g$ are pairwise Borel isomorphic.
        \item \label{outerCompressible}
            If $G$ is nontrivial and $E\cong_B E\oplus(E\times I_\N)$,
            then $E$ is compressible.
    \end{enumerate}
\end{prop}
\begin{proof}
    For \ref{bireducibleIso},
    suppose $E_g$ lives on $X_g$,
    and let $F$ be a CBER on $Y$ such that
    $F\sim_B E_g$ for every $g\in G$,
    and for each $g\in G$,
    fix a Borel bijection $f_g:Y/F\to X_g/E_g$.
    Define $G\car_B X/E$ for $[x]_E\in X_g/E_g$ by
    $h\cdot[x]_E = f_{hg}(f_g^{-1}([x]_E))$.
    By assumption,
    this action is induced by some $G\to\Out_B(E)$,
    which induces isomorphisms between the $E_g$.
    
    For \ref{outerCompressible},
    since $E\cong_B E\oplus (E\times I_\N)$,
    by working in the cardinal algebra of
    (Borel isomorphism classes of) CBERs,
    we have $E\cong_B E\oplus\bigoplus_{g\in G\setminus \{1\}}(E\times I_\N)$.
    So by \ref{bireducibleIso},
    we have $E\cong_B E\times I_\N$.
\end{proof}
So if $E$ is non-compressible and satisfies $E\cong_B E\oplus (E\times I_\N)$,
then every nontrivial countable group admits a non-outer action on $X/E$.
There are many such examples:
\begin{eg}\label{nonOuter}
    \leavevmode
    \begin{enumerate}[label=(\arabic*)]
        \item \label{nonOuterE0}
            (Miller)
            We have $E_0\cong E_0\oplus (E_0\times I_\N)$,
            since they are both uniquely ergodic and hyperfinite.
            More generally $E\cong_B E\oplus (E\times I_\N)$,
            for any aperiodic hyperfinite CBER $E$.
        \item A countable group $G$ is \textbf{dynamically compressible}
            if every aperiodic orbit equivalence relation of $G$
            is Borel reducible to a compressible
            orbit equivalence relation of $G$.
            Examples include amenable groups,
            and groups containing a non-abelian free group.
            If $G$ is dynamically compressible,
            then $E^{\ap}(G,\R)\cong_B
            E^{\ap}(G, \R)\oplus(E^{\ap}(G,\R)\times I_\N)$.
            where $E^{\ap}(G,\R)$ denotes the aperiodic part
            of the shift action of $G$ on $\R^G$;
            see \cite[5(B)]{FKSV21}.
    \end{enumerate}
\end{eg}

\subsection{Lifts of compressible CBERs}
Every action $G\car_B X/E$ induces a CBER $E^{\vee G}\supseteq E$
defined as follows:
\[
    x\E^{\vee G} y
    \iff \exists g\in G(g\cdot [x]_E = [y]_E).
\]
Every action $G\car_B (X,E)$ induces an action $G \car_B X/E$,
and we write $E^{\vee G}$ for the CBER induced by the latter.
Note that $E^{\vee G} = E\vee E_G^X$.
If $G$ is a subgroup of $\Aut_B(E)$ or $\Out_B(E)$,
we write $E^{\vee G}$ for the CBER given by the (outer) action induced by the inclusion map,
and if $T \in \Aut_B(E)$,
we write $E^{\vee T}$ for $E^{\vee \ev T}$.

In \cite{dRM21},
it is shown that there is a countable basis
of pairs $E \subseteq F$ of CBERs
such that there is no Borel action $G \car_B X/E$
with $F = E^{\vee G}$
(see \Cref{embeddingsOfQuotients} for a precise statement).

Given $f \in \Sym_B(X/E)$,
a \textbf{lift} of $f$ is a map $T \in \Aut_B(E)$
such that $[T (x)]_E = f( [x]_E)$ for every $x \in X$.
Given an action $G \car_B X/E$,
a \textbf{lift} of $g \in G$ is a lift of its image in $\Sym_B(X/E)$.

The following notion is from \cite{Tse13}:
\begin{defn}\label{defnLink}
    Let $E\subseteq F$ be CBERs.
    An \textbf{$(E,F)$-link} is a CBER $L\subseteq F$
    such that for every $F$-class $C$,
    every $E\uhr C$-class meets every $L\uhr C$-class exactly once.
\end{defn}

The connection to lifts is the following:
\begin{prop}\label{linkLift}
    Let $G\car_B X/E$.
    Then the following are equivalent:
    \begin{enumerate}[label=(\arabic*)]
        \item \label{link}
            There is an $(E,E^{\vee G})$-link.
        \item \label{cb}
            There is a class-bijective lift $G\car_B (X,E)$.
    \end{enumerate}
\end{prop}
\begin{proof}
    \ref{cb} $\implies$ \ref{link}
    $E_G^X$ is a link.
    
    \ref{link} $\implies$ \ref{cb}
    Let $g\cdot x$ be the unique element in $[x]_L\cap (g\cdot[x]_E)$.
\end{proof}

\Cref{outerIso} perhaps suggests that if $E$ is compressible,
then every Borel action on $X/E$ is outer.
It turns out that something much stronger is true:
\begin{thm}\label{compressibleCb}
    Let $E$ be a compressible CBER.
    Then every Borel action on $X/E$ has a class-bijective lift.
\end{thm}
By \Cref{linkLift},
it suffices to prove the following,
independently established using a different method by Ben Miller
(see comments following \Cref{everyActionEquiv} below for his approach):
\begin{thm}\label{compressibleLink}
    Let $E\subseteq F$ be compressible CBERs.
    Then there is a smooth $(E,F)$-link.
\end{thm}

We will repeatedly use the following,
where we identify a positive integer $N$ with $\{0, 1, \ldots, N - 1\}$.
\begin{lem}\label{amplification}
    Let $E\subseteq F$ be compressible CBERs and let
    $N \in \{1, 2, \ldots, \N\}$.
    Then $(E, F)$ is Borel isomorphic to $(E \times I_N, F \times I_N)$,
    in symbols $(E, F) \cong_B (E \times I_N, F \times I_N)$,
    i.e.,
    there is a Borel isomorphism that takes
    $E$ to $E \times I_\N$ and $F$ to $F \times I_\N$.
\end{lem}
\begin{proof}
    Since $E$ is compressible, $E\cong_B E\times I_\N$.
    So  $(E, F)$ is Borel isomorphic to $(E\times I_\N, R)$,
    for some $R$,
    which then must be of the form $F'\times I_\N$.
    Thus $(E, F) \cong_B (E\times I_\N, F'\times I_\N)$,
    and therefore
    $(E \times I_N, F \times I_N)
    \cong_B (E \times I_\N \times I_N, F' \times I_\N \times I_N)
    \cong_B (E \times I_\N, F' \times I_\N)
    \cong_B (E, F)$,
    since $I_\N \cong_B I_\N \times I_N$.
\end{proof}

\begin{proof}[Proof of \Cref{compressibleLink}]
    We can assume that every $F$-class contains exactly $N$ $E$-classes,
    where $N \in \{1, 2, \ldots, \N\}$.
    Below, $i < N$ means $i\in N$.
    
    Fix a Borel action of a countable group $\Gamma$ generating $F$.
    
    Fix a \textbf{choice sequence} for $(E,F)$,
    that is,
    a sequence $(f_i)_{i<N}$ of Borel maps $X\to X$
    such that for every $x\in X$,
    the function $i\mapsto [f_i(x)]_E$
    is a bijection from $N$ to $[x]_F/E$.
    For instance,
    define $f_i$ inductively
    by setting $f_0(x) = x$ and $f_i(x) = \gamma \cdot x$,
    where $\gamma$ is least (in some enumeration of $G$) such that
    $\gamma\cdot x$ is not $E$-related to any $f_j(x)$ for $j < i$.
    
    We can assume that each $f_i$ is injective.
    By \Cref{amplification},
    it suffices to define an injective choice sequence
    for $(E\times I_\N, F\times I_\N)$.
    Fix a pairing function $\ev{-,-}:\N\times\Gamma\to \N$.
    Then we take the choice sequence for $(E\times I_\N, F\times I_\N)$
    defined by $(x,n)\mapsto (f_i(x),\ev{n,\gamma})$,
    where $f_i$ is a choice sequence for $(E, F)$ and
    $\gamma$ is least such that $\gamma \cdot x = f_i(x)$.
    
    We can further assume that each $\im f_i$ is a complete $E$-section.
    To see this,
    endow $N$ with some group operation $\star$,
    and take the choice sequence for $(E\times I_N, F\times I_N)$
    defined by $(x,k)\mapsto (f_{i\star k}(x),k)$,
    where $(f_i)$ is a choice sequence for $(E,F)$ with each $f_i$ injective.
    
    Moreover, we can assume that each $\im f_i$ is $E$-compressible.
    To see this,
    take the choice sequence for $(E\times I_\N, F\times I_\N)$
    defined by $(x,n)\mapsto (f_i(x),n)$,
    where $(f_i)$ is a choice sequence for $(E,F)$,
    with each $f_i$ injective and $\im f_i$ a complete $E$-section.
    
    Finally,
    we can assume that each $f_i$ is bijective.
    To see this,
    since $\im f_i$ is an $E$-compressible complete section for $E$,
    there is some Borel injection $T_i$ such that $T(x) \E x$ for every $x$,
    and $T_i(X) = \im f_i$.
    Then $(T_i^{-1}\circ f_i)$ is a choice sequence for $(E,F)$
    with each $T_i^{-1}\circ f_i$ bijective.

    Now we can define a smooth $(E\times I_N, F\times I_N)$-link $L$ as follows:
    \[
        (x,i) \mr L (y,j)
        \iff f_i^{-1}(x) = f_j^{-1}(y)
    \]
    and we are done again by \Cref{amplification}.
\end{proof}

\begin{cor}\label{everyActionEquiv}
    Let $E$ be an aperiodic CBER satisfying
    $E\cong_B E\oplus (E\times I_\N)$
    (for instance,
    any aperiodic hyperfinite CBER).
    Then the following are equivalent:
    \begin{enumerate}[label=(\arabic*)]
        \item \label{everyActionCb}
            Every Borel action on $X/E$ has a class-bijective lift.
        \item \label{everyActionLift}
            Every Borel action on $X/E$ has a lift.
        \item \label{everyActionOuter}
            Every Borel action on $X/E$ is outer.
        \item \label{everyActionOuterGroup}
            There is a nontrivial countable group $G$ such that
            every action $G\car_B X/E$ is outer.
        \item \label{everyActionCompressible}
            $E$ is compressible.
    \end{enumerate}
\end{cor}
\begin{proof}
    \ref{everyActionCb} $\implies$ \ref{everyActionLift}
    Immediate.
    
    \ref{everyActionLift} $\implies$ \ref{everyActionOuter}
    Immediate.
    
    \ref{everyActionOuter} $\implies$ \ref{everyActionOuterGroup}
    Immediate.
    
    \ref{everyActionOuterGroup} $\implies$ \ref{everyActionCompressible}
    Follows from \Cref{outerIso}.
    
    \ref{everyActionCompressible} $\implies$ \ref{everyActionCb}
    Follows from \Cref{compressibleCb}.
\end{proof}

Concerning \Cref{compressibleLink},
Ben Miller derives this from the following more general result
whose proof uses Proposition 4.1 and 4.2 from \cite{Mil18}.
\begin{thm}[Miller]
    Let $E$ and $F$ be compressible CBERs on $X$ and $Y$ respectively,
    and let $f \colon X/E \to Y/F$ be Borel.
    Then the following are equivalent:
    \begin{enumerate}
        \item $f$ is smooth-to-one,
            i.e.,
            for every $y \in Y$,
            the restriction of $E$ to
            $\{x \in X : f([x]_E) = [y]_F\}$ is smooth.
        \item There is a Borel function $T \colon X \to Y$
            such that for every $x \in X$,
            the restriction $T\uhr [x]_E$
            is a bijection from $[x]_E$ to $f([x]_E)$.
    \end{enumerate}
\end{thm}

However,
one only needs the special case where $f$ is countable-to-one.
Applying this to the case where $E \subseteq F$ and $f([x]_E) = [x]_F$,
we find a Borel map $T \colon X \to X$ such that
$T\uhr [x]_E$ is a bijection from $[x]_E$ to $[x]_F$.
Then we can define the link $L$ by $x \mr L y \iff T(x) = T(y)$.

To show generic lifting,
we need a strengthening of generic compressibility,
whose proof is a simple modification of the proof of
\cite[Corollary 13.3]{KM04}.
A more general version appears in \cite[Theorem 11.1]{Mil17}.
We include a proof for the reader's convenience.

\begin{thm}\label{comeagerCompressible}
    Let $E\subseteq F$ be aperiodic CBERs on a Polish space $X$.
    Then there is a comeager $F$-invariant,
    $E$-compressible Borel subset of $X$.
\end{thm}
\begin{proof}
    Fix a Borel coloring $c \colon [E]^{<\infty} \to \N$
    of the intersection graph.
    Write $X = \bigsqcup_{n\in\N} A_n$,
    where each $A_n$ is a Borel set meeting every $E$-class infinitely often;
    for instance,
    write $X = \bigsqcup_{(n,m)\in\N^2} B_{n,m}$,
    where each $B_{n,m}$ is a complete $E$-section
    (see \cite[1.2.6]{CM17}),
    and take $A_n = \bigcup_m B_{n,m}$.
    Let $\N^{<\N}$ denote the set of finite strings in $\N$.
    For $s \in \N^{<\N}$,
    let $\len(s)$ denote the length of $s$.
    For $s, t \in \N^{<\N}$,
    we write $s \preceq t$ to mean that $s$ is a prefix of $t$.
    We define fsr's $\{E_s\}_{s\in\N^{<\N}}$ of $E$ such that
    \begin{enumerate}[label=(\roman*)]
        \item if $s\preceq t$,
            then $E_s\subseteq E_t$,
        \item $A_0$ is a transversal for $E_s$,
        \item every $E_s$-class is contained in $\bigsqcup_{k\le\len(s)} A_k$.
    \end{enumerate}
    We proceed by induction on the length of $s$.
    Let $E_\varnothing$ be the equality relation on $A_0$.
    Now for each $a\in A_0$,
    let $[a]_{E_{s\concat i}}$ be the unique set,
    if it exists,
    of the form $[a]_{E_s}\sqcup S$,
    where $S\in [E]^{<\infty}$ is contained in $A_{\len(s)+1}$
    and $c([a]_{E_s}\sqcup S) = i$,
    and otherwise set $[a]_{E_{s\concat i}} = [a]_{E_s}$.
    This defines an fsr $E_s$ with the desired properties.
    
    For every $\alpha\in\N^\N$,
    let $E_\alpha = \bigcup_n E_{\alpha\uhr n}$.
    We claim that for every $a\in A_0$,
    we have
    \[
        \forall^*\alpha \,
        (\text{$[a]_{E_\alpha}$ is infinite}),
    \]
    where $\forall^*\alpha \, \Phi(\alpha)$ means that the set
    $\{\alpha\in\N^\N : \Phi(\alpha)\}$ is comeager
    (see \cite[8.J]{Kec95}).
    It suffices to show that for every $n$,
    we have
    \[
        \forall^*\alpha \,
        (|[a]_{E_\alpha}| > n).
    \]
    Since the set $\{\alpha\in\N^\N:|[a]_{E_\alpha}| > n\}$ is open,
    it suffices to show that it is dense.
    Fix some $s\in\N^{<\N}$.
    Let $S\in [E]^{<\infty}$ be a subset of $A_{\len(s)+1}$ with $|S| > n$.
    Then if $c([a]_{E_s}\sqcup S) = i$,
    then for every $\alpha\succ s\concat i$,
    we have $|[a]_{E_\alpha}| \ge |[a]_{E_{s\concat i}}| > n$,
    so we are done.
    
    Thus for every $x\in X$,
    we have
    \[
        \forall a\in A_0\cap [x]_F \,
        \forall^*\alpha \,
        (\text{$[a]_{E_\alpha}$ is infinite}),
    \]
    or equivalently
    \[
        \forall^*\alpha \, \forall a\in A_0\cap [x]_F \,
        (\text{$[a]_{E_\alpha}$ is infinite}),
    \]
    so by the Kuratowksi-Ulam theorem \cite[8.K]{Kec95},
    we have
    \[
        \forall^*\alpha \, \forall^* x \, \forall a\in A_0\cap [x]_F \,
        (\text{$[a]_{E_\alpha}$ is infinite}),
    \]
    so in particular,
    there is some $\alpha\in\N^\N$
    such that the $F$-invariant set
    \[
        C := \{x\in X:\forall a\in A_0\cap [x]_F \,
        (\text{$[a]_{E_\alpha}$ is infinite})\}
    \]
    is comeager.
    Note that $C$ is $E$-compressible,
    since $\dom(E_\alpha)\cap C$ is an $(E\uhr C)$-compressible,
    complete $(E\uhr C)$-section,
    so we are done.
\end{proof}
\begin{cor}\label{genericCb}
    Let $E$ be an aperiodic CBER on a Polish space $X$.
    Then for any Borel action $G\car_B X/E$,
    there is a comeager $E^{\vee G}$-invariant
    Borel subset $Y\subseteq X$ such that
    $G\car_B Y/E$ has a class-bijective lift.
\end{cor}
\begin{proof}
    Apply \Cref{comeagerCompressible} with $F = E^{\vee G}$.
    Then the result follows from \Cref{compressibleCb}.
\end{proof}

In conclusion,
let us say that an aperiodic CBER $E$ is \textbf{outer}
if every $G\car_B X/E$ is outer,
or equivalently $i_E$ is a bijection.
We have seen that every compressible CBER is outer,
while there are non-outer CBER.
However we have the following problems:
\begin{prob}\label{probOuterCBER}
    \leavevmode
    \begin{enumerate}[label=(\arabic*)]
        \item \label{outerNoncompressible}
            Are there outer, non-compressible CBER?
        \item Characterize the outer CBERs.
    \end{enumerate}
\end{prob}

Concerning the first part of this problem,
we note the following possible approach to finding such an example:

Assume that there is a a free,
pmp action of a countable group $G$
on a standard probability space $(X,\mu)$
with the following properties:
\begin{enumerate}[label=(\roman*)]
    \item \label{cohopfian}
        $G$ is co-Hopfian
        (i.e., injective morphisms of $G$ into itself are surjective)
        and $G$ has no non-trivial finite normal subgroups
        (e.g., $\SL_3(\Z))$,
    \item \label{totallyErgodic}
        the action is totally ergodic
        (i.e., every infinite subgroup acts ergodically)
        and satisfies cocycle superrigidity
        (i.e., every cocycle of the action to a countable group
        is cohomologous to a homomorphism),
    \item \label{outTrivial}
        $\Out_\mu(E_G^X)$ is trivial.
\end{enumerate}

There are many examples that satisfy \ref{totallyErgodic}
and others that satisfy \ref{outTrivial}
but it does not seem to be known whether there are examples that satisfy both.
Assuming that such an action exists,
one can see that the first part of the above problem has a positive answer.

By going to a $G$-invariant Borel set,
we can assume that $\mu$ is the unique invariant measure for this action.
Then if $Z\subseteq X$ is Borel and $G$-invariant of measure 1,
we have that $Y = X \setminus Z$ is compressible.
Put $E = E_G^X$.
Let now $f\in \Sym_B(X/E)$ and
let $T\colon X\to X$ be Borel such that $f([x]_E) = [T(x)]_E$.
Then $T$ is a reduction of $E$ to $E$ and so
it gives rise to a cocycle $\alpha$ of this action into $G$,
which is therefore cohomologous to a homomorphism $\varphi \colon G \to G$.
Thus we can find another Borel map $S$ with $S(x) \E T(x)$ and
$S(g\cdot x) = \varphi(g)\cdot S(x)$, a.e.
Let $N = \ker(\varphi)$.
If it is not trivial, it must be infinite.
Then for $g\in N$, $S(g\cdot x) = S(x)$, a.e.,
so by the ergodicity of the $N$-action,
$S$ is constant, a.e.,
which is a contradiction.
So $N$ is trivial and thus $\varphi$ is injective,
therefore an automorphism.
It follows that $S$ is in $\Aut_\mu(E)$ and thus in $\Inn_\mu (E)$.
Therefore there is an $E$-invariant Borel set $Z\subseteq X$ of measure 1
with $f \uhr (Z/E)$ the identity.
Then $f \uhr (Z/E)$ can be lifted to the identity of $Z$.
Moreover $Y = X\setminus Z$ is compressible,
so,
by \Cref{compressibleCb}
$f \uhr (Y/E)$ can be lifted to some Borel automorphism of $E \uhr Y$.
Thus $f$ is an outer permutation.

Concerning the second part of the problem,
note that by \Cref{everyActionEquiv},
an aperiodic hyperfinite CBER is outer iff it is compressible.

The following problem about the algebraic structure
of these groups is also open:
\begin{prob}\label{probOutNormal}
    When is $\Out_B(E)$ a normal subgroup of $\Sym_B(X/E)$?
\end{prob}

\section{Outer actions}
\label{sectionOuter}

A lift of an outer action
is a solution to the following lifting problem:
\[
    \begin{tikzcd}
        & \Aut_B(E) \dar[two heads, "p_E"] \\
        G \rar \urar[dashed] & \Out_B(E)
    \end{tikzcd}
\]

Many outer actions arise from the following construction:
\begin{eg}\label{outerExample}
    Given a Borel action $G \car X$ of a countable group $G$
    and a normal subgroup $N \tl G$,
    there is a morphism $G \to \Out_B(E_N^X)$ defined by
    \[
        g\cdot [x]_{E_N^X}
        = [g\cdot x]_{E_N^X},
    \]
    and this descends to a morphism $G/N \to \Out_B(E_N^X)$.
\end{eg}

\subsection{Normal subequivalence relations}

The concept of normality is central to the study of outer actions:
\begin{defn}
    Let $E\subseteq F$ be CBERs.
    We say that $E$ is \textbf{normal} in $F$,
    denoted $E\tl F$,
    if any of the following equivalent conditions hold:
    \begin{enumerate}[label=(\arabic*)]
        \item \label{actionNormal}
            There is an action $G \car_B (X,E)$
            of a countable group $G$ such that $F = E^{\vee G}$.
        \item \label{outerActionNormal}
            There is a morphism $G \to \Out_B(E)$
            from a countable group $G$ such that $F = E^{\vee G}$.
        \item \label{subgroupNormal}
            There is a countable subgroup $G \le \Aut_B(E)$
            such that $F = E^{\vee G}$.
        \item \label{outerSubgroupNormal}
            There is a countable subgroup $G \le \Out_B(E)$
            such that $F = E^{\vee G}$.
    \end{enumerate}
\end{defn}
To see the equivalence,
note that \ref{subgroupNormal} $\implies$
\ref{actionNormal} $\implies$ \ref{outerActionNormal} is immediate,
\ref{outerActionNormal} $\implies$ \ref{outerSubgroupNormal}
holds by taking the image of $G$ in $\Out_B(E)$,
and \ref{outerSubgroupNormal} $\implies$ \ref{subgroupNormal}
holds by fixing a lift $T_g\in\Aut_B(E)$ of each $g\in G$
and taking the subgroup of $\Aut_B(E)$ generated by the $T_g$.

For CBERs $E \subseteq F$,
it is possible that $E$ is not normal in $F$,
but that there is still a Borel action $G \car_B X/E$
such that $F = E^{\vee G}$,
as witnessed by the example
at the beginning of \Cref{outerNonOuter}.
For more discussion concerning the weaker notion,
see \Cref{embeddingsOfQuotients}.
\begin{prop}\label{normalFacts}
    Let $E\tl F$ be CBERs on $X$.
    \begin{enumerate}[label=(\arabic*)]
        \item \label{intermediateNormal}
            If $F'$ is a CBER with $E\subseteq F'\subseteq F$,
            then $E\tl F'$.
        \item \label{subsetNormal}
            For any $E$-invariant subset $Y\subseteq X$,
            we have $E\uhr Y\tl F\uhr Y$.
    \end{enumerate}
\end{prop}
\begin{proof}
    Note that \ref{subsetNormal} follows immediately from \ref{intermediateNormal}
    by taking $F' = (F\uhr Y)\oplus (F\uhr (X\setminus Y))$,
    so it suffices to prove \ref{intermediateNormal}.
    
    We first assume that $F = E^{\vee T}$ for some $T \in \Aut_B(E)$.
    We will show that $F' = E^{\vee T'}$ for some $T'\in \Aut_B(E)$.
    
    For each $x\in X$,
    let $\le_x$ be the preorder on $[x]_{F'}/E$
    defined by $[y]_E \le_x [z]_E$
    iff there exists some $n \ge 0$ such that $T^n(y) \E z$.
    If $\le_x$ is isomorphic to $\Z$ or not antisymmetric,
    then set $T'(x) = T^n(x)$,
    where $n > 0$ is least such that $T^n(x) \F' x$.
    Otherwise,
    there is a unique isomorphism from $\le_x$
    to either the negative integers $(\{\cdots, -3, -2, -1\}, \le)$
    or to an initial segment of $(\N, \le)$.
    So by fixing a transitive $\Z$-action on each of these linear orders,
    we obtain a transitive $\Z$-action on $[x]_{F'}/E$,
    and we set $T'(x) = T^n(x)$,
    where $n$ is unique such that $T^n(x) \in 1\cdot [x]_E$.
    
    Now suppose that $F = E^{\vee G}$ for some $G \le \Aut_B(E)$.
    By above,
    for each $T \in G$,
    we can fix some $T' \in \Aut_B(E)$
    such that $E^{\vee T'} = F'\cap E^{\vee T}$.
    Then $F' = E^{\vee H}$,
    where $H = \ev{T'}_{T \in G}$.
\end{proof}
We next make some remarks about smooth links.
Let $E\tl F$ be CBERs.
Suppose that $E$ is aperiodic and $[F:E] = \infty$,
since the finite parts have smooth links via the forthcoming 
\Cref{finiteLink} and \Cref{smoothFacts}.
If $E$ is compressible,
then there is a smooth link by \Cref{compressibleLink}.
On the other hand,
if there is a smooth link $L$,
then $F$ must be compressible,
since it contains the aperiodic smooth $L$.

Thus the existence of a link does not imply
the existence of a smooth link.
For instance,
fix a free pmp Borel action $\Z^2 \car X$,
and consider $E = E_{\Z\times \{0\}}^X$ and $F = E_{\Z^2}^X$.
Then there is a link given by the action of $\{0\} \times \Z$,
but there is no smooth link,
since $F$ is not compressible.
If $X$ is the circle and the $\Z^2$-action is by
two linearly independent irrational rotations,
then $E$ and $F$ are both uniquely ergodic,
and by taking copies of these,
one can obtain an example with any number of ergodic measures.

If $E\tl F$ with $E$ finitely ergodic,
then $F$ is not compressible,
since if $\EINV_E = (e_i)_{i < n}$,
then $\frac{1}{n}(e_0 + \cdots + e_{n-1}) \in \EINV_F$.
Thus there is no smooth link.
If $\EINV_E$ is infinite,
it is still possible for a smooth link to exist.
For instance,
consider $E = E_0 \times \Delta_\N$
and $F = E_0 \times I_\N$.
In general,
the following is open:
\begin{prob}\label{probSmoothLink}
    Let $E\tl F$ be CBERs with $F$ is compressible.
    Is there a smooth $(E,F)$-link?
\end{prob}

Another open question, related to \Cref{compressibleLink}, is as follows:

\begin{prob}\label{probCompressibleExt}
    Let $E\tl F\tl F'$ be compressible CBERs.
    Can every $(E, F)$-link be extended to an $(E, F')$-link?
\end{prob}
If this were true,
then assuming the Continuum Hypothesis,
for any compressible CBER $E$,
the epimorphism $p_E \colon \Aut_B(E) \thra \Out_B(E)$ would split,
i.e.,
there would exist a morphism $s \colon \Out_B(E) \to \Aut_B(E)$
with $p_E \circ s$ equal to the identity.
To see this,
write $\Out_B(E)$ as an increasing union
$\bigcup_{\alpha < \omega_1} G_\alpha$ of countable subgroups.
It suffices to obtain class-bijective lifts
$G_\alpha \to \Aut_B(E)$ such that if $\alpha < \beta$,
then the $G_\beta$ lift extends the $G_\alpha$ lift.
For $\lambda$ limit,
take the union of the corresponding links for the $G_\alpha$ with $\alpha < \lambda$,
and for $\beta = \alpha + 1$ a successor,
use a positive answer to \Cref{probCompressibleExt}.

\subsection{Basic results}

\begin{prop}\label{smoothFacts}
    Let $E$ be a smooth CBER.
    \begin{enumerate}[label=(\arabic*)]
        \item \label{smoothLink}
            If $F$ is a CBER with $E\tl F$,
            then there is an $(E,F)$-link.
        \item \label{smoothCb}
            Every outer action on $X/E$ has a class-bijective lift.
    \end{enumerate}
\end{prop}
\begin{proof}
    By \Cref{linkLift},
    it suffices to show \ref{smoothLink}.
    
    By normality,
    any two $E$-classes contained in the same $F$-class have the same cardinality,
    so by partitioning the space into $F$-invariant Borel sets,
    we can assume that there is some
    $n \in \{1, 2, \cdots, \N\}$
    such that every $E$-class has cardinality $n$.
    Then there is a partition $X = \bigsqcup_{k < n} S_k$
    such that each $S_k$ is a transversal for $E$.
    Thus the CBER $L$ defined by
    \[
        x \mr L y\iff (x \F y) \amp
        (\exists k < n \, [x, y\in S_k])
    \]
    is an $(E,F)$-link.
\end{proof}

It is clear that if $G$ is a free group,
then every outer action of $G$ has a lift.
There are also some basic closure properties for the class of groups
for which every outer action admits a (class-bijective) lift.
\begin{prop}\label{subgroupClosed}
    Let $H\le G$.
    If every outer action of $G$ has a (class-bijective) lift,
    then the same holds for $H$.
\end{prop}
\begin{proof}
    Let $E$ be a CBER,
    and fix a morphism $H\to \Out_B(E)$.
    Let $F = \bigoplus_{G/H} E$.
    Then there is a morphism $G\to\Out_B(F)$,
    induced by the action of $G$ on $G/H$,
    so we get a lift $G\to \Aut_B(F)$.
    Restricting to $H$ and $E$ gives the desired lift.
\end{proof}

\begin{prop}\label{quotientClosed}
    Let $G\thra H$ be an epimorphism.
    If every outer action of $G$ has a class-bijective lift,
    then the same holds for $H$.
\end{prop}
\begin{proof}
    Fix a morphism $H\to \Out_B(E)$.
    This gives a morphism $G \to \Out_B(E)$.
    Since by surjectivity $E^{\vee G} = E^{\vee H}$,
    we are done by \Cref{linkLift}.   
\end{proof}

At this point,
it is good to show that not every outer action has a lift.
\begin{defn}
    A countable group $G$ is \textbf{treeable}
    if it admits a free pmp Borel action
    whose induced equivalence relation is treeable.
\end{defn}
\begin{eg}
    There are many examples of groups which are not treeable
    (see \cite[30]{KM04}, \cite[10.8]{Kec22}):
    \begin{itemize}
        \item Infinite property (T) groups.
        \item $G\times H$, where $G$ is infinite and $H$ is non-amenable.
        \item More generally,
            lattices in products of locally compact Polish groups $G\times H$,
            where $G$ is non-compact and $H$ is non-amenable.
    \end{itemize}
\end{eg}

The proof of the next result is motivated by \cite[Theorem 5]{CJ85}
and the remark following the proof of \cite[Theorem 3.4]{FSZ89}.
\begin{prop}\label{liftTreeable}
    Suppose that every outer action of $G$ lifts.
    Then $G$ is treeable.
\end{prop}
\begin{proof}
    We can assume that $G = F_\infty/N$ for some $N \tl F_\infty$,
    where $F_\infty$ is the free group on infinitely many generators.
    Fix a free pmp Borel action $F_\infty \car_B (X, \mu)$
    (for instance,
    the Bernoulli shift on $2^{F_\infty}$),
    and consider the induced free outer action $G \to \Out_B(E_N^X)$
    (see \Cref{outerExample}).
    By assumption,
    there is a lift $G \to \Aut_B(E_N^X)$,
    which is also a free action.
    Then $E_G^X$ is treeable and preserves $\mu$,
    since $E_{F_\infty}^X$ satisfies these properties
    and contains $E_G^X$.
\end{proof}

Note that we have no control over the treeable CBER in
the proof of \Cref{liftTreeable}.
In particular,
the following is open:
\begin{prob}\label{probE0Lift}
    Does every outer action on $X/E_0$ lift?
\end{prob}

\section{Outer actions of finite groups}

The following is a strengthening of \cite[Proposition 7.1]{Tse13}:
\begin{thm}\label{finiteLink}
    Let $E\tl F$ be a finite index extension of CBERs.
    Then there is an $(E,F)$-link.
\end{thm}
\begin{proof}
    Let $\Phi$ be the set of elements of $[F]^{<\infty}$
    which are a transversal for $E\uhr C$ for some $F$-class $C$.
    By \cite[Lemma 7.3]{KM04},
    there is a $\Phi$-maximal fsr $R$.
    Let $Y = (\dom(R))_E$ be the $E$-hull of $\dom(R)$.
    
    Let $G\le \Aut_B(E)$ be a countable subgroup such that $F = E^{\vee G}$.
    For every $x\in X\setminus Y$,
    let $g_x\in G$ be least (in some enumeration of $G$)
    such that $g_x\cdot x\in Y$;
    this exists by $\Phi$-maximality of $R$.
    Then the equivalence relation generated by $R\uhr Y$
    and $\{(x,g_x\cdot x):x\in X\setminus Y\}$ is an $(E,F)$-link.
\end{proof}

\begin{cor}\label{finiteCb}
    Every outer action of a finite group has a class-bijective lift.
\end{cor}
\begin{proof}
    Follows from \Cref{linkLift} and \Cref{finiteLink}.
\end{proof}

The following is a special case of \Cref{amenableCb},
whose proof is much harder.
\begin{cor}\label{ZCb}
    Every outer action of $\Z$ has a class-bijective lift.
\end{cor}
\begin{proof}
    On the finite $\Z$-orbits,
    apply \Cref{finiteCb}.
    On the infinite $\Z$-orbits of $X/E$,
    just lift uniquely.
\end{proof}

We next introduce lifts of morphisms:
\begin{defn}
    Let $H\to G$ be a morphism of countable groups.
    Then $H\to G$ has the \textbf{class-bijective lifting property}
    if for any CBER $E$ and any diagram of the form
    \[
        \begin{tikzcd}
            H \dar \rar & \Aut_B(E) \dar[two heads, "p_E"] \\
            G \rar & \Out_B(E)
        \end{tikzcd}
    \]
    with $H\to \Aut_B(E)$ class-bijective,
    there is a class-bijective lift $G\to \Aut_B(E)$.
\end{defn}

\begin{prop}\label{amalgamLift}
    Let $H$ be a countable group,
    let $(G_n)_n$ be a countable family of countable groups,
    let $H\to G_n$ be morphisms,
    and let $G$ be the amalgamated free product of the $G_n$ over $H$.
    If every outer action of $H$ has a class-bijective lift,
    and each $H \to G_n$ has the class-bijective lifting property,
    then every outer action of $G$ lifts.
\end{prop}
\begin{proof}
    Let $E$ be a CBER,
    and fix $G \to \Out_B(E)$.
    By assumption,
    there is a class-bijective lift of $H \to \Out_B(E)$.
    Then for each $n$,
    there is a class-bijective lift $G_n \to \Aut_B(E)$
    such that the following diagram commutes:
    \[
        \begin{tikzcd}
            H \dar \rar & \Aut_B(E) \dar[two heads, "p_E"] \\
            G_n \rar \urar & \Out_B(E)
        \end{tikzcd}
    \]
    Thus by the universal property of amalgamated products,
    there is a lift $G \to \Aut_B(E)$.
\end{proof}

\begin{thm}\label{finiteNormalSubgroup}
    Let $G$ be a countable group
    and let $N\tl G$ be a finite normal subgroup
    such that every outer action of $H = G/N$ has a class-bijective lift.
    \begin{enumerate}[label=(\arabic*)]
        \item \label{finiteNormalInclusion}
            The inclusion $N\hra G$ has the class-bijective lifting property.
        \item \label{finiteNormalCb}
            Every outer action of $G$ has a class-bijective lift.
    \end{enumerate}
\end{thm}
\begin{proof}
    \ref{finiteNormalInclusion} implies \ref{finiteNormalCb} by \Cref{finiteCb},
    so it suffices to show \ref{finiteNormalInclusion}.
    
    Let $E$ be a CBER on $X$,
    and suppose we have
    \[
        \begin{tikzcd}
            N \dar[hook] \rar & \Aut_B(E) \dar[two heads, "p_E"] \\
            G \rar & \Out_B(E)
        \end{tikzcd}
    \]
    with $N\to \Aut_B(E)$ class-bijective,
    and let $F = E^{\vee N}$.
    Note that $L = E_N^X$ is an $(E, F)$-link.
    There is an induced outer action $H \to \Out_B(F)$.
    We can assume that $[F:E] = n < \infty$.
    Let $S$ be a transversal for $L$,
    and fix a Borel action $\Z/n\Z \car X$ generating $L$.
    
    Define an injection $\Aut_B(F\uhr S)\hra\Aut_B(F)$ as follows:
    given $T \in \Aut_B(F\uhr S)$,
    let $T' \in \Aut_B(F)$ be the unique morphism satisfying
    $T'(k\cdot x) = k\cdot T(x)$ for every $x\in S$ and $k \in \Z/n\Z$.
    This descends to an injection $\Out_B(F\uhr S) \hra \Out_B(F)$
    satisfying the following commutative diagram:
    \[
        \begin{tikzcd}
            \Out_B(F \uhr S) \rar[hook] \dar[hook, "i_{F\uhr S}"]
            & \Out_B(F) \dar[hook, "i_F"] \\
            \Sym_B(F \uhr S) \rar["\cong"] & \Sym_B(F)
        \end{tikzcd}
    \]
    We claim that this injection is a bijection.
    To see this,
    let $T\in \Aut_B(F)$.
    Since $X = \bigsqcup_{k \in \Z/n\Z} k\cdot S$,
    we have $n\wt S = \wt X$
    in the cardinal algebra $\K(F\times I_\N)$.
    Thus $n\wt{T(S)} = \wt{T(X)} = \wt X$,
    so by the cancellation law,
    we have $\wt S = \wt{T(S)}$,
    i.e.,
    there is some $T'\in \Inn_B(F)$ with $T'(T(S)) = S$.
    Then $(T'T)\uhr S\in \Aut_B(F\uhr S)$ is the desired map.
    
    Thus we obtain an outer action $H\to\Out_B(F\uhr S)$
    and by assumption,
    there is an $(F\uhr S, E^{\vee G}\uhr S)$-link $L'$.
    Then the equivalence relation generated by $L$ and $L'$ is an $(E,F')$-link.
\end{proof}

We will prove next a generalization of \Cref{finiteCb} to morphisms.
For that,
we need the following result.
\begin{prop}\label{boundedIndex}
    Let $E\subseteq F$ be a bounded index extension of CBERs.
    Then the following are equivalent:
    \begin{enumerate}[label=(\arabic*)]
        \item \label{boundedNormal}
            $E\tl F$.
        \item \label{finiteSubgroup}
            There is a finite subgroup $G\le\Out_B(E)$
            such that $F = E^{\vee G}$.
    \end{enumerate}
\end{prop}
\begin{proof}
    \ref{finiteSubgroup} $\implies$ \ref{boundedNormal} Immediate.
    
    \ref{boundedNormal} $\implies$ \ref{finiteSubgroup}
    Let $H = (h_n)_n\le \Aut_B(E)$ be a countable subgroup
    such that $F = E^{\vee H}$.
    We define inductively a sequence
    $(g_n)_n\subseteq \Inn_B(F)\cap \Aut_B(E)$ as follows:
    for every $F$-class $C$,
    if there is $i$ such that
    $p_{E\uhr C}(h_i\uhr C)\neq p_{E\uhr C}(g_j\uhr C)$ for all $j < n$,
    then for the least $i$ with this property,
    set $g_n\uhr C = h_i\uhr C$;
    otherwise set $g_n\uhr C = \id{}\uhr C$.
    
    Note that the sequence $(g_n)_n$ is eventually equal to $\id_X$,
    since $E$ is of bounded index in $F$.
    Thus the group $\tilde G = \ev{g_n}_{n < \infty} \le \Inn_B(F)\cap \Aut_B(E)$
    is finitely generated.
    Note also that $F = E^{\vee\tilde G}$.
    Now the image of $\Inn_B(F)\cap \Aut_B(E)$ in $\Out_B(E)$ is locally finite,
    since it is a subgroup of $(S_n)^{X/F}$ for some finite symmetric group $S_n$.
    So the image $G$ of $\tilde G$ in $\Out_B(E)$ is finite,
    and we are done.
\end{proof}

We have a generalization of \Cref{finiteLink}:
\begin{thm}\label{finiteLinkExt}
    Let $E\subseteq F\subseteq F'$ be CBERs such that
    $E$ has finite index in $F'$ and $E\tl F'$.
    Then every $(E,F)$-link is contained in an $(E,F')$-link.
\end{thm}
\begin{proof}
    By partitioning the underlying standard Borel space $X$,
    we can assume that there is some $n < \infty$ such that
    every $F'$-class contains at most $n$ $F$-classes.
    We proceed by induction on $n$.
    The case $n = 1$ is trivial.
    
    Let $L$ be an $(E,F)$-link and let $S$ be a transversal for $L$.
    Let $\Phi$ be the set of $A\in [F'\uhr S]^{<\infty}$
    which are a transversal for $F\uhr C$ for some $F'$-class $C$.
    By \cite[Lemma 7.3]{KM04},
    there is a $\Phi$-maximal fsr $R$.
    Let $Y\subseteq X$ be the set of $x\in X$
    such that $[x]_F\subseteq [\dom(R)]_L$
    and let $Z = X\setminus Y$.
    We can assume that no $F'$-class is contained in $Y$,
    since the equivalence relation generated by $R$ and $L$
    is an $(E,F')$-link on such a class.
    By $\Phi$-maximality of $R$,
    no $F'$-class is contained in $Z$ either.
    By \ref{subsetNormal} of \Cref{normalFacts},
    we have $E\uhr Y\tl F'\uhr Y$,
    so by the induction hypothesis,
    there is an $(E\uhr Y, F'\uhr Y)$-link $L_Y$ containing $L\uhr Y$.
    Similarly,
    there is an $(E\uhr Z, F'\uhr Z)$-link $L_Z$ containing $L\uhr Z$.
    
    Let $S_Y$ and $S_Z$ be transversals for $L_Y$ and $L_Z$ respectively.
    It suffices to show that there is some $T\in \Inn_B(F')$
    such that $T(S_Y) = S_Z$,
    since then the smallest equivalence relation
    containing $L_Y$ and $L_Z$ and $\{(x, T(x)):x\in S_Y\}$
    is an $(E,F')$-link.
    In other words,
    we need to show that $\wt{S_Y} = \wt{S_Z}$
    in the cardinal algebra $\K(F'\times I_\N)$.
    By \Cref{boundedIndex},
    there is a finite subgroup $G\le\Out_B(E)$
    such that $F' = E^{\vee G}$.
    By partitioning $X$,
    we can assume that $[F'\uhr Y:E\uhr Y] = n_Y$
    and $[F'\uhr Z: E \uhr Z] = n_Z$
    for some $n_Y, n_Z < \infty$.
    Then $\wt Y = n_Y\wt{S_Y}$ and $\wt Z = n_Z\wt{S_Z}$.
    Let $k = \frac{|G|}{n_Y + n_Z}$.
    Then for every $x\in X$,
    we have
    \[
        |\{g \in G : [x]_E \subseteq g\cdot Y\}|
        = \sum_{[y]_E \subseteq Y}
            |\{g \in G : [x]_E = g\cdot [y]_E\}|
        = k n_Y,
    \]
    and thus $|G|\wt Y = k n_Y\wt X$.
    Similarly, $|G|\wt Z = k n_Z\wt X$.
    Thus
    \[
        |G| n_Y n_Z \wt{S_Y}
        = |G| n_Z \wt Y
        = k n_Y n_Z \wt X
        = |G| n_Y \wt Z
        = |G| n_Y n_Z \wt{S_Z},
    \]
    which yields $\wt{S_Y} = \wt{S_Z}$ by the cancellation law.
\end{proof}

\begin{cor}\label{finiteHomCb}
    Every morphism of finite groups has the class-bijective lifting property.
\end{cor}
\begin{proof}
    Suppose we have
    \[
        \begin{tikzcd}
            H \dar \rar & \Aut_B(E) \dar[two heads, "p_E"] \\
            G \rar & \Out_B(E)
        \end{tikzcd}
    \]
    with $H$ and $G$ finite,
    and $H\to \Aut_B(E)$ class-bijective.
    Then $E_H$ is an $(E,E^{\vee H})$-link,
    so by \Cref{finiteLinkExt},
    there is an $(E,E^{\vee G})$-link $L_G$ containing $E_H$.
    This lets us define an action of $G$
    by setting $g\cdot x$ to be the unique element
    in both $[x]_{L_G}$ and $g\cdot [x]_E$.
\end{proof}

\begin{cor}\label{finiteAmalgamLift}
    Every outer action of an amalgamated free product
    of finite groups has a lift.
\end{cor}
\begin{proof}
    Let $H$ be a finite group,
    let $(G_n)_{n < \infty}$ be finite groups,
    let $H \to G_n$ be morphisms,
    and let $G$ be the amalgamated free product of the $G_n$ over $H$.
    By \Cref{finiteCb},
    every outer action of $H$ has a class-bijective lift.
    By \Cref{finiteHomCb},
    the morphisms $H \to G_n$ have the class-bijective lifting property.
    Thus by \Cref{amalgamLift},
    every outer action of $G$ lifts.
\end{proof}

Given CBERs $E\subseteq F$,
we say that $F/E$ is \textbf{hyperfinite}
if there is an increasing sequence $(F_n)_n$
of finite index extensions of $E$ such that $F = \bigcup_n F_n$.
\begin{cor}\label{hfLink}
    Let $E\tl F$ be CBERs with $F/E$ hyperfinite.
    Then there is an $(E,F)$-link.
\end{cor}
\begin{proof}
    Apply \Cref{finiteLinkExt} countably many times.
\end{proof}

\begin{cor}\label{lfCb}
    Every outer action of a locally finite group
    has a class-bijective lift.
\end{cor}
\begin{proof}
    Immediate from \Cref{hfLink}.
\end{proof}

\section{Outer actions of amenable groups}
\label{sectionAmenable}

Our goal in this section is to show that
every outer action of an amenable group lifts.
We will prove in \ref{specialCases} some special cases of this result,
using (as a black box) \cite[Theorem 3.4]{FSZ89}
(stated in \Cref{amenablePmp} below).
The general case,
which is based on some ideas from the proof of
\Cref{amenablePmp} in combination with \Cref{compressibleCb}
will be proved in \ref{generalCase}.

\subsection{Special cases}
\label{specialCases}

We will use the following result from the pmp setting:
\begin{thm}[ {\cite[Theorem 3.4]{FSZ89}}]\label{amenablePmp}
    Let $G$ be an amenable group and let $E$ be a pmp ergodic CBER.
    Then any morphism $G\to\Out_\mu(E)$ has a lift.
\end{thm}

\begin{rmk}
    In \cite{FSZ89} this result is stated for free outer actions,
    i.e., outer actions $\varphi\colon G\to\Out_\mu(E)$
    that have the following additional property:
    if $g\in G$ is not the identity and $T_g\in \Aut_\mu (E)$
    maps by the canonical projection to $\varphi (g)$,
    then $T_g (x) \notin [x]_E$, a.e.
    Using the ergodicity of $E$,
    this is equivalent to the kernel of $\varphi$ being trivial.
    Thus for an arbitrary outer action $\varphi\colon G\to\Out_\mu(E)$,
    if $H$ is the kernel of $\varphi$,
    this gives a free outer action of $G/H$,
    which by the special case lifts to an action of $G/H$
    which composed with the projection of $G$ to $G/H$
    gives a lifting of $\varphi$.
\end{rmk}

\begin{rmk}
    Note that (the measurable version of)
    \Cref{finiteAmalgamLift}
    gives examples of non-amenable groups
    that satisfy \Cref{amenablePmp}.
\end{rmk}

Now \Cref{amenablePmp} together with \Cref{compressibleCb}
implies the following Borel result:
\begin{thm}\label{amenableLiftUnique}
    Let $G$ be an amenable group
    and let $E$ be a uniquely ergodic CBER.
    Then every morphism $G\to\Out_B(E)$ lifts.
\end{thm}
\begin{proof}
    Let $\mu$ be the ergodic invariant measure for $E$.
    Note that any element of $\Aut_B(E)$ preserves $\mu$ by unique ergodicity.
    Thus by \Cref{amenablePmp},
    there is a lift $G\to \Aut_\mu(E)$,
    so there is a conull $E$-invariant Borel set $Y\subseteq X$
    such that $G\to\Out_B(E\uhr Y)$ lifts to $\Aut_B(E\uhr Y)$.
    But since the complement is compressible,
    we are done here by \Cref{compressibleCb}.
\end{proof}

In fact the following stronger result holds.
\begin{thm}\label{amenableLiftCountable}
    Let $G$ be an amenable group
    and let $E$ be a countably ergodic CBER.
    Then every morphism $G\to\Out_B(E)$ lifts.
\end{thm}
\begin{proof}
    Note that $G$ acts on the ergodic components modulo compressible sets,
    which we can ignore by \Cref{compressibleCb}.
    We can assume that this action is transitive.
    Fix an ergodic component $Y$,
    and let $H = \{g\in G:g\cdot Y = Y\}$.
    By the uniquely ergodic case,
    there is a lift $H\to\Aut_B(E\uhr Y)$.
    Let $S\subseteq G$ be a transversal
    for the left cosets of $H$ in $G$,
    with $1\in S$.
    For every $s\in S$,
    choose a lift $T_s\in \Aut_B(E)$,
    with $T_1 = \id_X$.
    Now fix $g\in G$ and $s\in S$.
    We define the action of $g$ on $sY$.
    We have $gsY = tY$ for some $t\in S$,
    so we have $t^{-1}gs\in H$.
    Thus we can define
    \[
        g\cdot (T_s y) := T_t((t^{-1}gs)\cdot y).
    \]
\end{proof}

\subsection{$E$-null sets}
Let $E$ be an aperiodic CBER on $X$,
so that every $\mu \in \EINV_E$ is non-atomic.
A Borel subset $A\subseteq X$ is \textbf{$E$-null}
if either of the following equivalent conditions holds:
\begin{enumerate}[label=(\arabic*)]
    \item $\mu(A) = 0$ for every $\mu \in \EINV_E$.
    \item $E\uhr[A]_E$ is compressible.
\end{enumerate}
An \textbf{$E$-conull} set is the complement of an $E$-null set.

Let $\NULL_E \subseteq \B(X)$ be the $\sigma$-ideal of $E$-null Borel sets,
and let $\ALG_E$ be the quotient $\sigma$-algebra $\B(X)/\NULL_E$.
A Borel map $T \colon X \to X$ is \textbf{$\NULL_E$-preserving}
if the preimage under $T$ of every $E$-null set is $E$-null.
Let $\End_{\NULL_E}(E)$ be the monoid of
$\NULL_E$-preserving Borel maps $X\to X$
such that $x \E y \implies \varphi(x) \E \varphi(y)$
for all $x, y$ in an $E$-conull set,
where two such maps are identified if they agree on an $E$-conull set.
Let $\Aut_{\NULL_E}(E)$ be the group of invertible elements
of $\End_{\NULL_E}(E)$.
There is a natural action of $\Aut_{\NULL_E}(E)$ on $\ALG_E$.
Denote by $\Inn_{\NULL_E}(E)$ the normal subgroup of $\Aut_{\NULL_E}(E)$
of $\varphi$ such that $\varphi (x) \E x$ for an $E$-conull set of $x$,
and denote by $\Out_{\NULL_E}(E)$ the quotient group
$\Aut_{\NULL_E}(E)/\Inn_{\NULL_E}(E)$.

Lifts of elements of $\Out_{\NULL_E}(E)$
are defined analogously as in the case of $\Out_B(E)$,
as well as lifts of morphisms $G \to \Out_{\NULL_E}(E)$.
Let $G \to \Aut_{\NULL_E}(E)$ be a morphism.
Let $G \to \Out_{\NULL_E}(E)$.
There is an action on $X/E$ given by
\[
    g \cdot [x]_E = [T(x)]_E
\]
where $T$ is a lift of $g$,
which is well-defined for an $E$-conull set of $x$.
Then $\Stab_G([x]_E)$ is well-defined for an $E$-conull set of $x$.
We say that this is a \textbf{free action}
if $\Stab_G([x]_E) = 1$ for an $E$-conull set of $x$.
A morphism $G \to \Aut_{\NULL_E}(E)$ is \textbf{class-bijective}
if for every $g \in G$,
there is an $E$-conull set of $x$ such that
$\Stab_G(x) = \Stab_G([x]_E)$
(note that $\Stab_G(x)$ is also well-defined for an $E$-conull set of $x$).
Links are defined as before,
except that everything only needs to hold on an $E$-conull set.

Given $g\in\Out_{\NULL_E}(E)$,
a \textbf{partial lift} $\psi$ of $g$
is the restriction of a lift $\phi$ of $g$ to some $A \in \ALG_E$.
In this case, we write $\psi \colon A \to B$,
where $B = \phi(A)$.

There is a commutative diagram
\[
    \begin{tikzcd}
        1 \rar & \Inn_B(E) \rar \dar & \Aut_B(E) \rar \dar
        & \Out_B(E) \rar \dar & 1 \\
        1 \rar & \Inn_{\NULL_E}(E) \rar & \Aut_{\NULL_E}(E) \rar
        & \Out_{\NULL_E}(E) \rar & 1
    \end{tikzcd}
\]
In particular,
any morphism $G \to \Out_B(E)$
induces a morphism $G \to \Out_{\NULL_E}(E)$.
\begin{prop}\label{liftModNull}
    Let $E$ be an aperiodic CBER on $X$,
    let $G$ be a countable group
    and fix a morphism $G \to \Out_B(E)$.
    Then the following are equivalent:
    \begin{enumerate}[label=(\arabic*)]
        \item \label{actualLift}
            $G \to \Out_B(E)$ lifts.
        \item \label{almostLift}
            $G \to \Out_{\NULL_E}(E)$ lifts.
    \end{enumerate}
\end{prop}
\begin{proof}
    \ref{actualLift} $\implies$ \ref{almostLift}
    Immediate.
    
    \ref{almostLift} $\implies$ \ref{actualLift}
    Denote the lift by $\varphi:G\to\Aut_{\NULL_E}(E)$,
    and denote by $\varphi_g\in\Aut_{\NULL_E}(E)$
    the image of $g$ under $\varphi$.
    For each $g\in G$,
    pick a representative $T_g:X\to X$ of $\varphi_g$.
    There is an $E$-conull subset $Y \subseteq X$ such that
    \begin{enumerate}[label=(\roman*)]
        \item $x\E y \iff T_g(x) \E T_g(y)$
            for every $g \in G$ and $x,y \in Y$,
        \item $T_1(x) = x$ for every $x \in Y$,
        \item $T_g(T_h(x)) = T_{gh}(x)$
            for every $g,h \in G$ and $x \in Y$,
        \item $[T_g(x)]_E = g\cdot [x]_E$
            for every $g \in G$ and $x \in Y$.
    \end{enumerate}
    By taking the $E^{\vee G}$-hull,
    we can assume that $Y$ is $E^{\vee G}$-invariant.
    Then the $T_g$ define a lift of $G\to\Out_B(E\uhr Y)$.
    On $X\setminus Y$,
    we have that $E$ is compressible,
    so we are done by \Cref{compressibleCb}.
\end{proof}

Every $\mu \in \EINV_E$ is a well-defined measure on $\ALG_E$,
and there is an action $\Aut_{\NULL_E}(E) \car \EINV_E$ given by
\[
    (\varphi\cdot\mu)(A) = \mu(\varphi^{-1}(A)),
\]
which descends to an action of $\Out_{\NULL_E}(E)$.

\begin{prop}\label{equivalentSets}
    Let $E$ be an aperiodic CBER,
    let $g \in \Out_{\NULL_E}(E)$,
    and let $A, B \in \ALG_E$.
    Then the following are equivalent:
    \begin{enumerate}[label=(\arabic*)]
        \item \label{sameMeasure}
            $\mu(A) = (g\cdot \mu)(B)$
            for every $\mu \in \EINV_E$.
        \item \label{partialLift}
            There is a partial lift $\varphi:A\to B$ of $g$.
        \item \label{totalLift}
            There is a lift $\varphi$ of $g$ with $\varphi(A) = B$.
    \end{enumerate}
\end{prop}
\begin{proof}
    \ref{partialLift} $\iff$ \ref{totalLift}
    By definition.
    
    \ref{totalLift} $\implies$ \ref{sameMeasure}
    Immediate.
    
    \ref{sameMeasure} $\implies$ \ref{totalLift}
    Let $\psi$ be a lift of $g$.
    Then $\mu(A) = (g \cdot \mu)(B) = \mu(\psi^{-1}(B))$,
    so by replacing $B$ with $\psi^{-1}(B)$,
    we can assume that $g = 1$.
    Then the result follows from \cite[Lemma 7.10]{KM04}
    and the remark following it.
\end{proof}

A family $(\varphi_n)_n$ of partial maps is \textbf{disjoint}
if the family $(\dom \varphi_n)_n$ is disjoint
and the family $(\cod \varphi_n)_n$ is disjoint.
\begin{prop}\label{disjointLifts}
    Let $E$ be an aperiodic CBER,
    fix a morphism $G \to \Out_{\NULL_E}(E)$,
    and let $g \in G$.
    If $(\varphi_n)_n$ are disjoint partial lifts of $g$,
    then $\bigsqcup_n \varphi_n$ is a partial lift of $g$.
\end{prop}
\begin{proof}
    Suppose $\varphi_n \colon A_n \to B_n$.
    Let $A = X\setminus\bigsqcup_n A_n$
    and let $B = X\setminus\bigsqcup_n B_n$.
    By \Cref{equivalentSets},
    for any $\mu \in \EINV_E$,
    we have $\mu(A_n) = (g \cdot \mu)(B_n)$,
    and thus $\mu(A) = (g \cdot \mu)(B)$.
    So again by \Cref{equivalentSets},
    there is a partial lift $\varphi \colon A \to B$ of $g$.
    Then $\varphi \sqcup \bigsqcup_n \varphi_n$ is a lift of $g$,
    and thus the restriction $\varphi_n$ is a partial lift of $g$.
\end{proof}

For $A \in \ALG_E$,
we write $\mu_E(A) = r$
if for every $\mu \in \EINV_E$,
we have $\mu(A) = r$.
Recall that for any standard probability space $(X, \mu)$,
if $A \subseteq X$ and $r \le \mu(A)$,
then there is some $B \subseteq A$ with $\mu(A) = r$,
and this $B$ can be found uniformly in $\mu$.
By applying this to each $E$-ergodic component,
we obtain the following:
\begin{prop}\label{smallerSubset}
    Let $E$ be an aperiodic CBER,
    let $A \in \ALG_E$,
    and let $r \in [0, 1]$.
    If $r \le \mu_E(A)$,
    then there is some $B \subseteq A$ such that $\mu_E(B) = r$.
\end{prop}

\subsection{Quasi-tilings}

Let $G$ be a group.
Let $\Fin(G)$ denote the set of finite subsets of $G$,
and let $\Fin_1(G)$ denote the set of $A\in\Fin(G)$ containing $1$.
Given $A, B \in \Fin(G)$,
we say that $B$ \textbf{$\lambda$-covers} $A$
if $|A\cap B| \ge \lambda |A|$.

Let $\A$ be a family in $\Fin(G)$,
i.e., a subset of $\Fin(G)$.
We say that $\A$ is \textbf{$\varepsilon$-disjoint}
if there is a disjoint family $\{D_A\}_{A \in \A}$
such that each $D_A$ is a subset of $A$
which $(1 - \varepsilon)$-covers $A$.
Note that if $\A$ is $\varepsilon$-disjoint,
then
\[
    (1 - \varepsilon)\sum_{A\in \A}|A|
    \le \qty|\bigcup_{A\in \A} A|.
\]
Given $A\in \Fin(G)$,
we say that $\A$ \textbf{$\lambda$-covers} $A$
if $\bigcup_{B \in \A} B$ $\lambda$-covers $A$.

Let $\A$ be a family in $\Fin_1(G)$ and let $A \in \Fin(G)$.
An \textbf{$\A$-quasi-tiling} of $A$
is a tuple $\C = (C_B)_{B\in \A}$ of subsets of $A$
such that $Bc\subseteq A$ for every $c\in C_B$,
and the family $\{B C_B\}_{B\in \A}$ is disjoint.
If $1\in A$,
we additionally demand that $1\in C_B$ for some $B\in \A$.
If $\A = \{B\}$ is a singleton,
we will write ``$C$ is a $B$-quasi-tiling''
as shorthand to mean that $(C)$ is a $\{B\}$-quasi-tiling.
We say that $\C$ is \textbf{$\varepsilon$-disjoint}
if for each $B\in \A$,
the family $\{B c\}_{c\in C_B}$ is $\varepsilon$-disjoint.
We say that $\C$ \textbf{$\lambda$-covers} $A$
if $\{B C_B\}_{B \in \A}$ $\lambda$-covers $A$.
We say that $\C$ is an
\textbf{$(\A,\varepsilon)$-quasi-tiling} of $A$
if it is $\varepsilon$-disjoint and $(1 - \varepsilon)$-covers $A$.

Given $A \in \Fin(G)$ and $B \in \Fin_1(G)$,
let $T(A, B)$ denote the set $\{a\in A : Ba\subseteq A\}$.
We say that $A$ is \textbf{$(B, \varepsilon)$-invariant}
if $T(A, B)$ $(1 - \varepsilon)$-covers $A$.
Note that if $A$ is $(B, \varepsilon)$-invariant,
then $|B A| \le (1 + \varepsilon |B|)|A|$.
\begin{lem}\label{invariantCover}
    Let $G$ be group,
    let $\delta, \varepsilon > 0$,
    let $B\in \Fin_1(G)$,
    and let $A\in \Fin(G)$ be $(B, \delta)$-invariant.
    Then any maximal $\varepsilon$-disjoint family $\{Bc\}_{c\in C}$
    of right translates of $B$ contained in $A$
    $\varepsilon (1 - \delta)$-covers $A$.
\end{lem}
\begin{proof}
     If $g \in T(A, B)$,
     then by maximality,
     we have $|Bg \cap BC| \ge \varepsilon |B|$.
     Thus
     \[
        \varepsilon (1 - \delta)|A|
        \le \varepsilon |T(A, B)|
        \le \sum_{g \in T(A, B)}\frac{|Bg \cap BC|}{|B|}
        \le \sum_{g \in G}\frac{|Bg \cap BC|}{|B|}
        = |BC|,
     \]
     where the last equality holds since every element of $BC$
     is contained in exactly $|B|$-many right translates of $B$.
\end{proof}

Let $\A$ be a finite family in $\Fin_1(G)$
and let $\p = (p_B)_{B\in \A}$ be
a probability distribution on $\A$.
Given an $\A$-quasi-tiling $\C = (C_B)_{B\in \A}$ of $A \in \Fin(G)$,
we say that $\C$ \textbf{satisfies} $\p$ if
$|B||C_B| \le p_B|A|$ for every $B\in \A$.
Given $\varepsilon > 0$,
we say that the pair $(\A, \varepsilon)$ \textbf{satisfies} $\p$
if there is some $\delta > 0$ such that
for every $A\in\Fin_1(G)$ larger than $\frac{1}{\delta}$
which is $(B, \delta)$-invariant and contains $B$ for every $B\in \A$,
there is an $(\A, \varepsilon)$-quasi-tiling of $A$ satisfying $\p$.

\begin{lem}\label{quasitilingInduction}
    Let $G$ be a group.
    For every $\varepsilon > 0$,
    there is a finite probability distribution
    $\p = (p_i)_{i < k}$
    and constants $\eta_i > 0$ for $i < k - 1$
    such that if $\A = (B_i)_{i < k}$ is a
    descending chain in $\Fin_1(G)$ where
    each $B_i$ for $i < k - 1$ is
    $(B_{i+1}^{-1}, \frac{\eta_i}{|B_{i+1}|})$-invariant,
    then $(\A, \varepsilon)$ satisfies $\p$.
\end{lem}
\begin{proof}
    By scaling, it suffices to find a subprobability distribution.
    Choose $k$ such that $2\varepsilon \ge (1 - \varepsilon)^k$,
    define $p_i = \varepsilon (1 - \varepsilon)^i$,
    and for $i < k -1$,
    choose $\eta_i$ such that
    \[
        \eta_i \le \frac{1 - 2\varepsilon}{2\cdot 3^{k-i}}.
    \]
    
    Let $\A = (B_i)_{i < k}$ be a
    descending chain in $\Fin_1(G)$ where each $B_i$ is
    $(B_{i+1}^{-1}, \frac{\eta_i}{|B_{i+1}|})$-invariant,
    and let $\delta > 0$ be sufficiently small,
    depending on $(\A, \varepsilon)$,
    to be specified in the course of the proof.
    Suppose we have some $A\in\Fin_1(G)$ which is
    larger than $\frac{1}{\delta}$
    and $(B, \delta)$-invariant for every $B\in \A$.
    
    We define a descending sequence $(A_i)_{i < k}$ of subsets of $A$
    and $2\varepsilon$-disjoint $B_i$-quasi-tilings $C_i$ of $A_i$ such that
    \begin{enumerate}[label=(\roman*)]
        \item $A_0 = A$.
        \item \label{Ai}
            $A_{i+1} = A_i\setminus B_i C_i$,
        \item \label{AiInvariant}
            $A_i$ is $(B_i, \frac{1}{3^{k - i}})$-invariant,
        \item \label{ratioBiCiToA}
            \[
                \varepsilon (1 - \varepsilon)^{i + 2 - 2^{-i}}
                \le \frac{|B_i C_i|}{|A|}
                \le \varepsilon (1 - \varepsilon)^{i - 2 + 2^{-i}},
            \]
        \item \label{ratioAiToA}
            \[
                (1 - \varepsilon)^{i + 2 - 2^{-i + 1}}
                \le \frac{|A_i|}{|A|}
                \le (1 - \varepsilon)^{i - 2 + 2^{-i + 1}}.
            \]
    \end{enumerate}
    We proceed by induction,
    starting with $A_0 = A$,
    defining $C_i$ from $A_i$,
    and defining $A_{i+1}$ from $C_i$ via \ref{Ai}.
    Note that $A_0$ satisfies \ref{AiInvariant}
    if we require $\delta \le \frac{1}{3^k}$.
    
    Suppose that $A_i$ has been defined.
    We will define $C_i$.
    Let $\tilde C_i$ be a maximal $2\varepsilon$-disjoint
    $B_i$-quasi-tiling of $A_i$.
    Since $2\varepsilon \qty(1 - \frac{1}{3^{k - i}}) > \varepsilon$,
    by \Cref{invariantCover},
    $\tilde C_i$ is an $\varepsilon$-cover of $A_i$.
    Then by removing elements from $\tilde C_i$,
    we obtain a $B_i$-quasi-tiling
    $C_i\subseteq \tilde C_i$ of $A_i$
    such that
    \[
        \varepsilon (1 - \varepsilon)^{2^{-i}}
        \le \frac{|B_i C_i|}{|A_i|}
        \le \varepsilon (1 - \varepsilon)^{-2^{-i}}
    \]
    and
    \[
        (1 - \varepsilon)^{1 + 2^{-i}}
        \le \frac{|A_{i+1}|}{|A_i|}
        \le (1 - \varepsilon)^{1 - 2^{-i}},
    \]
    as long as $A_i$ is sufficiently large such that
    $\frac{|B_i|}{|A_i|}$ is smaller than
    the length of the interval around $\varepsilon$ given by
    \[
        \qty[\varepsilon (1 - \varepsilon)^{2^{-i}},
        \varepsilon (1 - \varepsilon)^{-2^{-i}}]
        \cap
        \qty[1 - (1 - \varepsilon)^{1 - 2^{-i}},
        1 - (1 - \varepsilon)^{1 + 2^{-i}}],
    \]
    which occurs for sufficiently large $A$ by \ref{ratioAiToA}.
    Then since $\frac{|B_i C_i|}{|A|}
    = \frac{|B_i C_i|}{|A_i|}\frac{|A_i|}{|A|}$,
    we get that \ref{ratioBiCiToA} holds.
    Similarly,
    \ref{ratioAiToA} holds for $A_{i+1}$.
    
    It remains to check \ref{AiInvariant}.
    Note that
    \[
        T(A_{i+1}, B_{i+1})
        = T(A_i, B_{i+1})
        \setminus B_{i+1}^{-1} B_i C_i.
    \]
    Since
    \[
        \frac{|A_{i+1}|}{|A_i|}
        \ge (1 - \varepsilon)^{1 + 2^{-i}}
        \ge (1 - \varepsilon)^2
        \ge \frac{1}{2},
    \]
    where we assume that $\varepsilon$ is small enough
    to satisfy the last inequality,
    the cardinality of $T(A_i, B_{i+1})$ is at least
    \[
        \qty(1 - \frac{1}{3^{k - i}}) |A_i|
        \ge |A_i| - \frac{2}{3^{k - i}} |A_{i + 1}|.
    \]
    Now $B_i C_i$ is
    $\qty(B_{i+1}^{-1},
    \frac{\eta_i}{|B_{i+1}|(1 - 2\varepsilon)})$-invariant,
    since
    \begin{align*}
        |\{g\in B_i C_i : B_{i + 1}^{-1} g\not\subseteq B_i C_i\}|
        & \le \sum_{c\in C_i}|\{g\in B_i c : B_{i + 1}^{-1} g\not\subseteq B_i C_i\}| \\
        & \le \sum_{c\in C_i}|\{g\in B_i c : B_{i + 1}^{-1} g\not\subseteq B_i c\}| \\
        & \le \sum_{c\in C_i}\frac{\eta_i}{|B_{i + 1}|}|B_i| \\
        & = \frac{\eta_i}{|B_{i + 1}|}|B_i||C_i| \\
        & \le \frac{\eta_i}{|B_{i + 1}|}
        \frac{|B_i C_i|}{1 - 2\varepsilon}.
    \end{align*}
    Since
    \[
        \frac{|A_{i+1}|}{|B_i C_i|}
        \ge \frac{|A_{i+1}|}{|A_i|}
        \ge \frac{1}{2}
        \ge \frac{\eta_i}{1 - 2\varepsilon}3^{k-i},
    \]
    we have
    \[
        |B_{i + 1}^{-1} B_i C_i|
        \le \qty(1 + \frac{\eta_i}{1 - 2\varepsilon})|B_i C_i|
        \le |B_i C_i| + \frac{1}{3^{k-i}}|A_{i+1}|.
    \]
    Putting these together,
    we get
    \[
        |T(A_{i+1}, B_{i+1})|
        \ge \qty(1 - \frac{3}{3^{k - i}})|A_{i+1}|,
    \]
    so \ref{AiInvariant} holds.
    This concludes the construction.
    
    Now
    \[
        \frac{|B_i C_i|}{|A|}
        \ge \varepsilon (1 - \varepsilon)^{i + 2 - 2^{-i}}
        > \varepsilon (1 - \varepsilon)^{i + 2}
        > \varepsilon (1 - 2\varepsilon)^2 (1 - \varepsilon)^i,
    \]
    so for each $i < k$,
    there is a $B_i$-quasi-tiling
    $C'_i\subseteq C_i$ of $A_i$
    such that
    \[
        \varepsilon (1 - 2\varepsilon)^2 (1 - \varepsilon)^i
        \le \frac{|B_i C'_i|}{|A|}
        \le \varepsilon (1 - 2\varepsilon)(1 - \varepsilon)^i,
    \]
    as long as $A$ is large enough such that
    $\frac{|B_i|}{|A|}$ is smaller than
    the length of the interval
    \[
        \qty[\varepsilon (1 - 2\varepsilon)^2 (1 - \varepsilon)^i,
        \varepsilon (1 - 2\varepsilon)(1 - \varepsilon)^i].
    \]
    Then $(C'_i)_{i < k}$ is a $2\varepsilon$-disjoint $\A$-quasi-tiling of $A$
    which $(1 - 2\varepsilon)^3$-covers $A$.
    We also have
    \[
        \frac{|B_i||C'_i|}{|A|}
        \le \frac{1}{1 - 2\varepsilon}\frac{|B_i C'_i|}{|A|}
        \le \varepsilon (1 - \varepsilon)^i
        = p_i.
    \]
    
    So we are done by replacing $\varepsilon$ in the above argument
    by any $\bar\varepsilon$ such that $\varepsilon$ is greater
    than $2\bar\varepsilon$ and $1 - (1 - 2\bar\varepsilon)^3$.
\end{proof}

A countable group $G$ is \textbf{amenable}
if for every $B\in \Fin(G)$ and every $\varepsilon > 0$,
there is some $A\in \Fin(G)$ which is $(B, \varepsilon)$-invariant.
Note that we can assume that $A$ contains $B$.

\begin{prop}\label{quasitiling}
    Let $G$ be an amenable group and let
    $(\varepsilon_n)_{n < \infty}$ be a sequence of positive reals.
    Then there exist for each $n < \infty$,
    a finite family $\A_n$ in $\Fin_1(G)$
    and a probability distribution
    $\p^n$ on $\A_n$ such that
    \begin{enumerate}[label=(\roman*)]
        \item $\A_0 = \{\{1\}\}$,
        \item if $B\in \A_n$ and $A\in \A_{n + 1}$,
            then $A$ is $(B, \varepsilon_n)$-invariant and contains $B$,
        \item every $A\in \A_{n+1}$ has an
            $(\A_n, \varepsilon_n)$-quasi-tiling satisfying $\p^n$,
        \item $G = \bigcup_n\bigcup_{B\in\A_n}B$.
    \end{enumerate}
\end{prop}
\begin{proof}
    Fix an enumeration $(g_n)_n$ of $G$.
    We inductively define $\A_n$ and $\p^n$ satisfying
    the given conditions such that additionally,
    $(\A_n, \varepsilon_n)$ satisfies $\p^n$.
    For $n = 0$,
    take $\A_0 = \{\{1\}\}$,
    and let $\p^0$ be the unique probability distribution on $\A_0$.
    Then $(\A_0, \varepsilon_0)$ satisfies $\p^0$.
    Now suppose that $\A_n$ and $\p^n$ have been defined.
    Apply \Cref{quasitilingInduction} to $\varepsilon_{n+1}$
    to obtain a probability distribution $\p^n = (p_i)_{i < k_n}$
    and constants $(\eta^n_i)_{i < k_n - 1}$.
    We turn to defining $\A_{n+1} = (B^{n+1}_i)_{i < k_{n+1}}$.
    First we define $B^{n+1}_{k_{n+1} - 1}$,
    by choosing any $B^{n+1}_{k_{n+1} - 1}\in\Fin_1(G)$
    which contains $B$ and is
    $(B, \varepsilon_n)$-invariant for every $B\in\A_n$,
    and contains $g_n$.
    and which has an $(\A_n, \varepsilon_n)$-quasi-tiling satisfying $\p^n$
    (which is possible since $(\A_n, \varepsilon_n)$ satisifies $\p^n$).
    Now for any $i < k_{n+1} - 1$,
    we define $B^{n+1}_i$ from $B^{n+1}_{i+1}$,
    by choosing any $B^{n+1}_i\in\Fin_1(G)$
    containing $B^{n+1}_{i+1}$
    which is $\qty((B^{n+1}_{i+1})^{-1},
    \frac{\eta^n_i}{|B^{n+1}_{i + 1}|})$-invariant,
    $(B, \varepsilon_n)$-invariant for every $B\in\A_n$,
    and which has an $(\A_n, \varepsilon_n)$-quasi-tiling satisfying $\p^n$.
    Then $\A_{n+1}$ satisfies the given conditions
    and additionally,
    $(\A_{n+1}, \varepsilon_{n+1})$ satisfies $\p^{n+1}$.
\end{proof}

\subsection{General case}\label{generalCase}

\begin{thm}\label{amenableLift}
    Every outer action of an amenable group lifts.
\end{thm}
\begin{proof}
    Let $G$ be an amenable group,
    and let $E$ be a CBER on $X$.
    By \Cref{smoothFacts},
    we can assume that $E$ is aperiodic.
    By \Cref{liftModNull},
    it suffices to show that every morphism
    $G\to\Out_{\NULL_E}(E)$ lifts to $\Aut_{\NULL_E}(E)$.
    For the rest of the proof,
    when we refer to a subset of $X$,
    we will mean its equivalence class in $\ALG_E$.
    
    Fix a sequence $(\varepsilon_n)_{n < \infty}$
    of positive reals less than $1$
    such that
    \[
        \sum_n (1-(1-\varepsilon_n)(1-3\varepsilon_n)) < \infty.
    \]
    Apply \Cref{quasitiling} to $(\varepsilon_n)_n$
    to obtain for each $n < \infty$,
    a finite family $\A_n$ in $\Fin_1(G)$
    and a probability distribution
    $\p^n = (p^n_A)_{A \in \A_n}$ on $\A_n$.
    For ease of notation,
    we will write $p_A$ instead of $p^n_A$.
    
    For each $n < \infty$,
    we construct a disjoint family
    $(X_A)_{A \in \A_n}\subseteq \ALG_E$,
    and partial lifts $\varphi^n_g \in \Aut_{\NULL_E}(E)$
    of some $g\in G$ such that
    \begin{enumerate}[label=(\roman*)]
        \item $\varphi^n_1 = \id_X$,
        \item for $A\in\A_n$,
            we have $|A|\mu_E(X_A) = p_A$,
        \item the family $\{\varphi^n_g(X_A) : A\in\A_n, g\in A\}$
            is disjoint,
        \item for $A\in\A_n$,
            if $g, h, gh\in A$,
            then $\varphi^n_{gh}$ and $\varphi^n_g\varphi^n_h$
            agree on $X_A$.
    \end{enumerate}
    We proceed by induction on $n$.
    For $n = 0$,
    take $X_{\{1\}} = X$ and $\varphi^0_1 = \id_X$.
    Now suppose that the construction holds for $n$.
    We will repeatedly use \Cref{equivalentSets},
    \Cref{disjointLifts},
    and \Cref{smallerSubset}
    to obtain the partial lifts $\varphi^{n+1}_g$.
    For each $A\in\A_{n+1}$,
    fix an $(\A_n, \varepsilon_n)$-quasi-tiling
    $(C^A_B)_{B\in \A_n}$ of $A$.
    By $\varepsilon_n$-disjointness,
    for each $B\in \A_n$
    there is a disjoint family
    $\{D^A_{B,c} c\}_{c\in C^A_B}$
    where each $D^A_{B,c}$ is a subset of $B$
    which $(1 - \varepsilon_n)$-covers $B$.
    For each $A\in\A_{n+1}$,
    choose $X_A\subseteq X_B$ where $1\in C^A_B$,
    such that $|A|\mu_E(X_A) = p_A$;
    we can do this since
    \[
        \frac{p_A}{|A|}
        \le \frac{|C^A_B|}{|A|}
        \le \frac{p_B}{|B|}
        = \mu_E(B).
    \]
    For each $A\in\A_{n+1}$,
    each $B\in\A_n$,
    and each $c\in C^A_B$,
    define $\varphi^{n+1}_c$ on $X_A$
    so that for every $B\in\A_n$,
    the family
    $\{\varphi^{n+1}_c(X_A): A\in\A_{n+1}, c\in C^A_B\}$
    is disjoint and contained in $X_B$
    (see \Cref{figure});
    we can do this since for each $A\in\A_{n+1}$,
    we have
    \[
        \sum_{c\in C^A_B} \mu_E(X_A)
        = |C^A_B| \frac{p_A}{|A|}
        \le p_A \frac{p_B}{|B|}
        = p_A \mu_E(X_B).
    \]
    Now for each $A\in\A_{n+1}$,
    each $B\in\A_n$,
    each $c\in C^A_B$,
    and each $h\in D^A_{B,c}$,
    define $\varphi^{n+1}_{hc}$ on $X_A$
    by setting it equal to $\varphi^n_h\varphi^{n+1}_c$.
    Then for each $A\in\A_{n+1}$ and each $g\in A$,
    define $\varphi^{n+1}_g$ on $X_A$ if it hasn't been already defined,
    such that the family
    $\{\varphi^{n+1}_g (X_A) : A\in\A_{n+1}, g\in A\}$
    partitions $X$;
    this is possible since
    \[
        \sum_{A \in \A_{n+1}}\sum_{g\in A} \mu_E(X_A)
        = \sum_{A \in \A_{n+1}} |A|\mu_E(X_A)
        = \sum_{A \in \A_{n+1}} p_A
        = 1.
    \]
    Finally,
    for each $A\in\A_{n+1}$ and $g, h, gh\in A$,
    define $\varphi^{n+1}_g$ on $\varphi^{n+1}_h(X_A)$
    by setting it to be equal to $\varphi^{n+1}_{gh}(\varphi^{n+1}_h)^{-1}$.
    This concludes the construction.
    
    \begin{figure}
        \centering
        \begin{tikzpicture}[scale=0.8]
            \def\h{7};
            \def\v{4};
            \def\i{3};
            \def\ii{4};
            \def\iii{3};
            \def\iv{4};
            \coordinate (top-1) at (-\h/2,\v);
            \coordinate (bot-1) at (-\h/2,-\v);
            \coordinate (top0) at (0,\v);
            \coordinate (bot0) at (0,-\v);
            \coordinate (top1) at (\h/2,\v);
            \coordinate (bot1) at (\h/2,-\v);
            \clip (0,0) ellipse [x radius = \h, y radius = \v];
            
            \begin{scope}
                \path[clip] (top-1)
                    to[relative, out=10, in=170] (bot-1)
                    -- (-\h,-\v)
                    -- (-\h,\v)
                    -- cycle;
                \draw[fill=lightgray] (-\h-1,-\v/3)
                to[relative, out=10, in=170] (-\h/2+1,-\v/3)
                -- (-\h/2+1,-\v) -- (-\h-1,-\v) -- cycle;
                \draw[dashed] (-\h-1,\v/3) to[relative, out=10, in=170] (-\h/2+1,\v/3);
            \end{scope}
            \begin{scope}
                \path[clip] (top-1)
                    to[relative, out=10, in=170] (bot-1)
                    -- (bot0)
                    to[relative, out=-10, in=-170] (top0)
                    -- cycle;
                \draw[fill=lightgray] (-\h/2-1,-\v/2)
                    to[relative, out=10, in=170] (1,-\v/2)
                    -- (1,-\v) -- (-\h/2-1,-\v) -- cycle;
                \draw[dashed] (-\h/2-1,0) to[relative, out=10, in=170] (1,0);
                \draw[dashed] (-\h/2-1,\v/2) to[relative, out=10, in=170] (1,\v/2);
            \end{scope}
            \begin{scope}
                \path[clip] (top0)
                    to[relative, out=10, in=170] (bot0)
                    -- (bot1)
                    to[relative, out=-10, in=-170] (top1)
                    -- cycle;
                \draw[fill=lightgray] (-1,-\v/3)
                    to[relative, out=10, in=170] (\h/2+1,-\v/3)
                    -- (\h/2+1,-\v) -- (-1,-\v) -- cycle;
                \draw[dashed] (-1,\v/3) to[relative, out=10, in=170] (\h/2+1,\v/3);
            \end{scope}
            \begin{scope}
                \path[clip] (top1)
                    to[relative, out=10, in=170] (bot1)
                    -- (\h, -\v)
                    -- (\h, \v)
                    -- cycle;
                \draw[fill=lightgray] (\h/2-1,-\v/2)
                to[relative, out=10, in=170] (\h+1,-\v/2)
                -- (\h+1,-\v) -- (\h/2-1,-\v) -- cycle;
                \draw[dashed] (\h/2-1,0) to[relative, out=10, in=170] (\h+1,0);
                \draw[dashed] (\h/2-1,\v/2) to[relative, out=10, in=170] (\h+1,\v/2);
            \end{scope}
            
            \draw (0,0) ellipse [x radius = \h, y radius = \v];
            \draw (top-1) to[relative, out=10, in=170] (bot-1);
            \draw (top0) to[relative, out=10, in=170] (bot0);
            \draw (top1) to[relative, out=10, in=170] (bot1);
            
            \def\r{\v/20};
            \def\s{\v/12};
            
            \draw (-\h*0.76, -\v*0.4) circle [radius = \r];
            \draw (-\h*0.7, -\v*0.55) circle [radius = \r];
            \draw (-\h*0.6, -\v*0.5) rectangle +(\s,\s);
            \draw (-\h*0.55, -\v*0.7) rectangle +(\s,\s);
            
            \draw[fill=black] (-\h*0.3, -\v*0.6) circle [radius = \r];
            \draw (-\h*0.25, -\v*0.78) circle [radius = \r];
            \draw (-\h*0.15, -\v*0.6) rectangle +(\s,\s);
            \draw (-\h*0.1, -\v*0.8) rectangle +(\s,\s);
            
            \draw (\h*0.2, -\v*0.4) circle [radius = \r];
            \draw (\h*0.12, -\v*0.6) circle [radius = \r];
            \draw (\h*0.2, -\v*0.8) circle [radius = \r];
            \draw (\h*0.3, -\v*0.4) rectangle +(\s,\s);
            \draw[fill=black] (\h*0.24, -\v*0.6) rectangle +(\s,\s);
            \draw (\h*0.3, -\v*0.8) rectangle +(\s,\s);
            
            \draw (\h*0.6, -\v*0.65) circle [radius = \r];
            \draw (\h*0.7, -\v*0.6) rectangle +(\s,\s);
        \end{tikzpicture}
        \caption{The shaded regions are $X_B$ for $B\in\A_n$,
        and the regions above each $X_B$ are
        its translates $\varphi^n_b(X_B)$ for $b\in B$.
        The black disk is some $X_A$,
        the other disks are its translates $\varphi^{n+1}_c(X_A)$,
        and analogously for the squares for some other $A'\in\A_{n+1}$.}
        \label{figure}
    \end{figure}
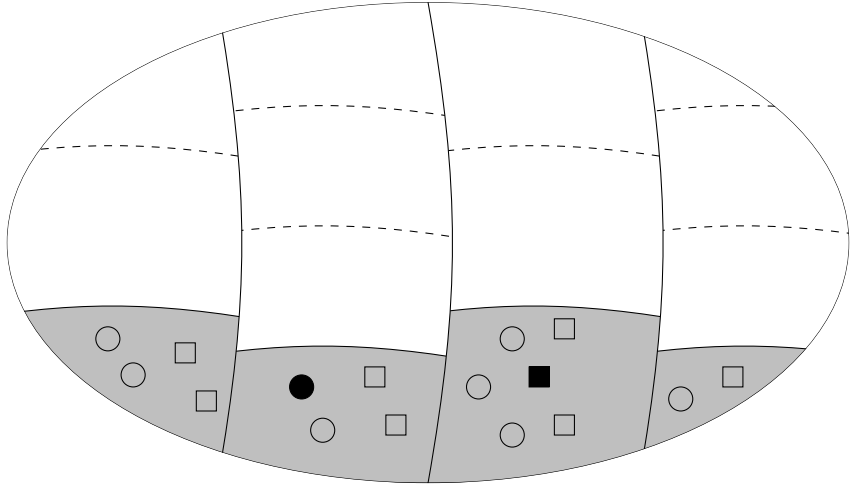
    
    We claim that for every $g\in G$,
    the pointwise limit $\varphi_g := \lim_n \varphi^n_g$ exists
    and is a total function.
    Let $n$ be large enough such that
    there is some $C \in \A_{n-1}$ with $g\in C$.
    Now for any $A \in \A_{n+1}$,
    $B \in \A_n$,
    $c \in C^A_B$,
    and $h \in D^A_{B,c}$ with $gh \in D^A_{B,c}$,
    we have on $X_A$,
    \[
        \varphi^n_g\varphi^{n+1}_{hc}
        = \varphi^n_g\varphi^n_h\varphi^{n+1}_c
        = \varphi^n_{gh}\varphi^{n+1}_c
        = \varphi^{n+1}_{ghc}
        = \varphi^{n+1}_g \varphi^{n+1}_{hc},
    \]
    so $\varphi^n_g$ and $\varphi^{n+1}_g$
    agree on $\varphi^{n+1}_{hc}(X_A)$.
    We have
    \[
        |B\setminus g^{-1} D^A_{B,c}|
        \le |B\setminus g^{-1} B|
        + |g^{-1}B\setminus g^{-1} D^A_{B,c}|
        < 2\varepsilon_n|B|.
    \]
    So $\varphi^n_g$ and $\varphi^{n+1}_g$ agree
    on a set of $\mu_E$-measure at least
    \begin{align*}
        \sum_{A\in\A_{n+1}}\sum_{B\in\A_n}\sum_{c\in C^A_B}
        \sum_{\substack{h\in D^A_{B,c}\\
        gh \in D^A_{B,c}}}
        \mu_E(\varphi^{n+1}_{hc}(X_A))
        & \ge \sum_{A\in\A_{n+1}}\sum_{B\in\A_n}|C^A_B|
        (1 - 3\varepsilon_n)|B|\frac{p_A}{|A|} \\
        & \ge \sum_{A\in\A_{n+1}}(1-\varepsilon_n)
        (1 - 3\varepsilon_n)p_A \\
        & \ge (1-\varepsilon_n)(1 - 3\varepsilon_n).
    \end{align*}
    So we are done by the Borel-Cantelli lemma.
    
    Now we claim that $g\mapsto\varphi_g$ is an action.
    Let $g, h\in G$.
    Choose $n$ large enough such that
    there is some $C\in\A_{n-1}$ with $g, h, gh\in C$.
    Now for any $B\in\A_n$
    and $k\in B$ with $hk, ghk \in B$,
    we have on $X_B$,
    \[
        \varphi^n_{gh} \varphi^n_k
        = \varphi^n_{ghk}
        = \varphi^n_g \varphi^n_{hk}
        = \varphi^n_g \varphi^n_h \varphi^n_k,
    \]
    so $\varphi^n_{gh}$ and $\varphi^n_g \varphi^n_h$
    agree on $\varphi^n_k(X_B)$.
    We have $|B\setminus h^{-1}B| \le \varepsilon_n|B|$
    and $|B\setminus (gh)^{-1}B| \le \varepsilon_n|B|$.
    So $\varphi^n_{gh}$ and $\varphi^n_g \varphi^n_h$
    agree on a set of $\mu_E$-measure at least
    \begin{align*}
        \sum_{B\in\A_n}
        \sum_{\substack{k\in B\\ hk, ghk\in B}}
        \mu_E(\varphi^n_k(X_B))
        & \ge \sum_{B\in\A_n}
        (1 - 2\varepsilon_n)|B|\mu_E(\varphi^n_r(X_B)) \\
        & \ge \sum_{B\in A_n}(1 - 2\varepsilon_n)p_B \\
        & \ge (1 - 2\varepsilon_n)
    \end{align*}
    So we are done by the Borel-Cantelli lemma.
\end{proof}

We can obtain class-bijective lifts for some amenable groups,
including abelian groups and amenable groups with countably many subgroups.
\begin{cor}\label{amenableCb}
    Let $G$ be an amenable group whose conjugacy equivalence relation
    on its space of subgroups is smooth.
    Then every outer action of $G$ has a class-bijective lift.
\end{cor}
\begin{proof}
    For this proof,
    we will work modulo $E$-null sets.
    Fix a morphism $G \to \Out_{\NULL_E}(E)$.
    Let $(X_e)_{e \in \EINV_E}$ be the ergodic decomposition of $E$.
    Let $\C$ be a transversal for the conjugacy equivalence relation
    on the space of subgroups,
    and for each subgroup $H\le G$,
    fix some $g_H \in G$ such that $g_H H g_H^{-1} \in \C$.
    The action $\Out_{\NULL_E}(E) \car \EINV_E$
    induces an action $G \car \EINV_E$.
    If $e \in \EINV_E$ has stabilizer $H \in \C$ under this action,
    then if $N_H$ is the kernel of $H \to \Out_{\NULL_E}(E \uhr X_e)$,
    we have $\Stab_H(x) = N_H$ by ergodicity,
    and thus $H/N_H \to \Out_{\NULL_E}(E \uhr X_e)$ is a free action.
    Thus by applying \Cref{amenableLift} to $X_e$,
    there is a class-bijective lift
    $H/N_H \to \Aut_{\NULL_E}(E \uhr X_e)$,
    and this gives a class-bijective lift
    $H \to \Aut_{\NULL_E}(E \uhr X_e)$,
    and thus a link.
    So for each $H \in \C$,
    if we let $X_H$ be the union of the ergodic components with stabilizer $H$,
    then there is an $(E \uhr X_H, E^{\vee H} \uhr X_H)$-link $L_H$.
    Now for an arbitrary subgroup $H \le G$,
    fix a lift $\psi_H$ of $g_H$.
    Then the smallest equivalence relation containing $L_H$
    and $\{(x, \psi_H(x)) :
    \text{$x \in X_e$ with $\Stab(e) = H$}\}$
    for every $H$
    is an $(E, E^{\vee G})$-link.
\end{proof}

\begin{rmk}
    There are locally finite groups for which
    the conjugacy equivalence relation on the space of subgroups is not smooth.
    Take, for example, a finite group $H$ with a non-normal subgroup $H'$ and let $\C$ be the conjugacy class of $H'$.
    Let $G = \bigoplus_n H$ be the infinite direct sum of copies of $H$.
    Consider the set $X$ of subgroups of $G$ of the form $\bigoplus_n H_n$,
    where $H_n \in \C$.
    Then $E_0$ is Borel reducible to the conjugacy equivalence relation on $X$,
    which is therefore non-smooth.
\end{rmk}

For general amenable groups,
the problem is still open:
\begin{prob}\label{probAmenableCb}
    Let $G$ be an amenable group.
    Does every $G \to \Out_B(E)$ have a class-bijective lift?
\end{prob}
We remark that in \Cref{probAmenableCb} it suffices to consider hyperfinite $E$.
To see this,
note that by \Cref{amenableLift},
there is a lift $G \to \Aut_B(E)$.
Then it suffices to find an $(E\cap E_G^X, E_G^X)$-link.
So by replacing $E$ with $E \cap E_G^X$,
we can assume that $E$ is amenable, in the sense of \cite[9.1]{Kec22},
and this is hyperfinite on an $E$-conull set, see \cite[9.4]{Kec22}.

\section{Summary of lifting results for outer actions}
\label{sectionSummary}

Let $\G$ be the class of groups for which
every outer action has a lift.
Then
\begin{itemize}
    \item $\G$ contains all amenable groups
        (\Cref{amenableLift}).
    \item $\G$ contains all amalgamated products of finite groups
        (\Cref{finiteAmalgamLift}).
    \item $\G$ is closed under subgroups
        (\Cref{subgroupClosed}).
    \item $\G$ is closed under free products.
    \item Every group in $\G$ is treeable
        (\Cref{liftTreeable}).
\end{itemize}

Let $\G_\cb$ be the class of groups for which
every outer action has a class-bijective lift.
Then
\begin{itemize}
    \item $\G_\cb$ contains all locally finite groups
        (\Cref{lfCb}).
    \item $\G_\cb$ contains all amenable groups
        whose conjugacy equivalence relation
        on the space of subgroups is smooth
        (\Cref{amenableCb}).
    \item $\G_\cb$ is closed under subgroups
        (\Cref{subgroupClosed}).
    \item $\G_\cb$ is closed under quotients
        (\Cref{quotientClosed}).
    \item $\G_\cb$ is closed under extensions
        by a finite normal subgroup
        (\Cref{finiteNormalSubgroup}).
\end{itemize}

\begin{prob}
    Characterize the classes $\G$ and $\G_\cb$.
\end{prob}

\section{Additional topics}
\label{sectionAdditional}

\subsection{Algebraic properties of automorphism groups}

There are several results concerning
the algebraic properties of $\Inn_B(E)$
(see \cite{Mil04}, \cite{Mer93}, \cite{MR07}),
and similarly for $\Inn_\mu(E)$ in the pmp case
(see \cite[\S \S 3-4]{Kec10} and the references therein).
In particular,
it is known that for aperiodic $E$,
the group $\Inn_B(E)$ is generated by involutions
and similarly for $\Inn_\mu(E)$.
However,
not much seems to be known about the groups
$\Aut_B(E), \Aut_\mu(E), \Out_B(E)$,
including the question about generation by involutions.
There are pmp, ergodic $E$ for which
$\Aut_\mu(E)$ is generated by involutions,
for example $E_0$ (see \cite[p.46]{Kec10})
and pmp ergodic $E$ that have trivial $\Out_\mu(E)$
(for the existence of such,
see \cite{Gef96}).
Since $E_0$ is uniquely ergodic,
the question of whether $\Aut_B(E_0)$ is generated by involutions
would have a positive answer
if $\Aut_B(E)$ is generated by involutions
for any hyperfinite compressible $E$.
So it seems natural to consider first the question of
generation by involutions of $\Aut_B(E)$,
where $E$ is a compressible CBER.

In the case of $\Sym_B(X/E)$,
Miller has shown that if $T \in \Sym_B(X/E)$
with $E^{\vee T}$ hyperfinite,
then $T$ is a product of three involutions.

\subsection{Conjugacy of outer actions}

A result of Bezuglyi-Golodets \cite{BG87},
in combination with \Cref{amenablePmp},
shows that any two morphisms
$\varphi_1, \varphi_2:G\to\Out_\mu(E_0)$ are conjugate
(i.e., there is $\theta\in \Out_\mu(E_0)$ such that
$\varphi_1(g) = \theta\varphi_2(g)\theta^{-1}$)
iff $\ker(\varphi_1) = \ker(\varphi_2)$.
Using \Cref{amenableLiftUnique},
one can see that the analogous result would hold
for morphisms of amenable groups into $\Out_B(E_0)$
if it holds for morphisms of amenable groups into $\Out_B(E)$
for $E$ compressible hyperfinite,
which again leads to the question of
whether an analog of the Bezuglyi-Golodets theorem holds
for morphisms of amenable groups into $\Out_B(E)$,
when $E$ is any compressible CBER.

\subsection{Embeddings of quotients}
\label{embeddingsOfQuotients}

For a countable group $G$,
let $F_0(G)$ be the CBER on $G^\N$ defined by
\[
    (g_0, g_1, g_2, \ldots) \mathrel{F_0(G)} (h_0, h_1, h_2, \ldots)
    \iff
    \exists m \,
    \forall k > m \,
    [g_0 \cdots g_k = h_0 \cdots h_k].
\]
There is an action $G \to \Aut_B(F_0(G))$ defined by
\[
    g\cdot (g_0, g_1, g_2, \ldots)
    = (g\cdot g_0, g_1, g_2, \ldots),
\]
inducing an action $G \car_B G^\N/F_0(G)$.
Given CBERs $E \subseteq F$ on $X$,
we say that $F/E$ is \textbf{ergodic}
if there is no Borel partition $X = A_0\sqcup A_1$
with each $A_i$ an $E$-invariant complete $F$-section.

Let $E$ be a CBER on a Polish space $X$,
and let $G \car_B X/E$ be a free action.
Then $E^{\vee G}/E$ is ergodic iff
there is a $G$-equivariant Borel injection
$G^\N/F_0(G) \hra X/E$
induced by a continuous embedding $G^\N \hra X$
(see \cite[Theorem 7.2]{Mil04}).
If $E^{\vee G}$ is hyperfinite,
then there is a $G$-equivariant Borel injection
$X/E \hra G^\N/F_0(G)$
(see \cite[Theorem 8.1]{Mil04}).

Given a pair $E \subseteq F$ of CBERs,
we say that $F/E$ is \textbf{generated by a Borel action}
if there is some Borel action $G \car_B X/E$
such that $F = E^{\vee G}$.
By \cite[Theorem 3]{Pin07},
this is equivalent to the existence
of a sequence of Borel functions
$f_n \colon X/E \to X/E$ such that
$x \F y \iff \exists n \, [f_n([x]_E) = [y]_E]$.
By \cite[Theorem 5]{dRM21},
there is a countable set of obstructions
for being generated by a Borel action.
Namely,
there is a sequence of pairs
$E_n \subseteq F_n$ of CBERs on $2^\N$
where $F_n/E_n$ is not generated by a Borel action,
such that if $E \subseteq F$ are CBERs on $X$
where $F/E$ is not generated by a Borel action,
then there is some $n$ for which
there is a continuous embedding $2^\N \hra X$
which simultaneously reduces $E_n$ to $E$
and $F_n$ to $F$.

\bibliography{liftsOfBorelActions}
\bibliographystyle{alpha}

\bigskip
\bigskip
\noindent Department of Mathematics

\noindent California Institute of Technology

\noindent Pasadena, CA 91125

\medskip
\noindent\textsf{joshfrisch@gmail.com, kechris@caltech.edu, fshinko@caltech.edu }

\end{document}